\documentclass[a4paper]{scrartcl}
\usepackage{todonotes}
\usepackage[utf8]{inputenc}
\usepackage{graphicx}
\usepackage{amsmath}  
\usepackage{amsfonts} 
\usepackage{amssymb}
\usepackage{amsthm}
\usepackage{enumitem}
\usepackage{nicefrac}
\usepackage{xstring}
\usepackage{xfrac}
\usepackage[hidelinks]{hyperref}
\usepackage[nameinlink,capitalize]{cleveref}

\crefdefaultlabelformat{#2\textup{#1}#3}
\title{l1 regularization}
\author{Philip Miller}
\date{\today\ }
\binoppenalty=\maxdimen
\relpenalty=\maxdimen
\setkomafont{disposition}{\rmfamily\bfseries\boldmath}   
\setkomafont{descriptionlabel}{\rmfamily\bfseries\boldmath}

\newtheorem{theorem}{Theorem}[section]
\newtheorem{lemma}[theorem]{Lemma}
\newtheorem{proposition}[theorem]{Proposition}
\newtheorem{corollary}[theorem]{Corollary}
\theoremstyle{remark} 
\newtheorem{remark}[theorem]{Remark}
\theoremstyle{definition}
 
\newtheorem{assumption}[theorem]{Assumption} 
\newtheorem{example}[theorem]{Example}
\newlist{thmlist}{enumerate}{1}
\setlist[thmlist]{label=\normalfont(\roman{thmlisti}),ref=\thetheorem(\roman{thmlisti})}
\theoremstyle{definition}

\usepackage{xstring}
\makeatletter
\newcommand{\refsub}[1]% #1 = label name for subtasks
{\@ifundefined{r@#1}{??}{\begingroup%
  \edef\temp{\expandafter\detokenize\getrefnumber{#1}}%
  \StrCut{\temp}{(}\temptask\tempsub%
  \edef\templink{\getrefbykeydefault{#1}{anchor}{}}%
  \if\thetask\temptask\hyperlink{\templink}{\tempsub)}%
  \else\hyperlink{\templink}{\temptask\,\tempsub)}\fi
\endgroup}}
\makeatother
\providecommand{\itemreffont}[1]{\textup{#1}}
\makeatletter
\newcommand{\itemref}[1]{%
  \begingroup
  \@ifundefined{r@#1}{\refused{#1}}{%
  \edef\tempx{\detokenize\getrefnumber{#1}}%  Fetch the literal reference 1.1(i) (here)
  \StrBetween{\tempx}{(}{)}[\@itemref]%  Get what's between the parantheses
  \edef\restore@@@@link{\getrefbykeydefault{#1}{anchor}{}}%
  \hyperlink{\restore@@@@link}{\itemreffont{(\@itemref)}}% Display the reference
  }%
\endgroup
}
\makeatother
\DeclareMathOperator*{\argmin}{argmin}
\DeclareMathOperator{\spn}{span}
\DeclareMathOperator{\supp}{supp}
\newcommand{\inv}{^{-1}}             %inversion
\newcommand{\gobs}{g^\mathrm{obs}}
\newcommand{\linOp}{T}
\newcommand{\nat}{\mathbb{N}_0} 
\newcommand{\pframe}{\phi}
\newcommand{\dframe}{\tilde{\phi}}

\newcommand{\loss}{l}
\newcommand{\err}{\mathbf{err}}
\newcommand{\Xspace}{\mathbb{X}}
\newcommand{\Yspace}{\mathbb{Y}}
\newcommand{\Omegaobs}{\Omega_{\Yspace}}
\newcommand{\norm}[2]{\left\|#1\,|\,#2\right\|}
\newcommand{\snorm}[1]{\left\|#1\right\|}
\begin{document}
\begin{center}
\textbf{\huge Optimal convergence rates for sparsity promoting 
wavelet-regularization in Besov spaces}\\[1ex]
\textbf{Thorsten Hohage and Philip Miller}
\end{center}
\begin{abstract}
This paper deals with Tikhonov regularization for linear and nonlinear ill-posed 
operator equations with wavelet Besov norm penalties. We focus on $B^0_{p,1}$ 
penalty terms which yield estimators that are sparse with respect to a wavelet frame. 
Our framework includes among others, the Radon transform and some nonlinear 
inverse problems in differential equations with distributed measurements.
Using variational source conditions it is shown that such estimators achieve 
minimax-optimal rates of convergence for finitely smoothing operators 
in certain Besov balls both for deterministic and for statistical noise models. 
\end{abstract}

\textbf{keywords:} {sparsity, wavelets, regularization, convergence rates, 
converse results, Besov spaces}

\section{Introduction}
We study the numerical solution of inverse problems formulated as (possibly nonlinear)  
ill-posed operator equations 
\[ F(f^\dagger)= g^\dagger 
\] 
in Banach spaces $\Xspace$ and $\Yspace$. 
Here $F\colon D\subset \Xspace\rightarrow \Yspace$ denotes the forward operator
mapping the unknown exact solution $f^{\dagger}$ to the data $g^{\dagger}$, which 
is defined on a non-empty subset $D$. 
A regularization method which has turned out to be particularly successful in 
numerical experiments is $1$-homogeneous penalization by wavelet coefficients 
which promotes sparsity (\cite{lorenz:08} and \Cref{theo:sparsity}). The purpose of 
this paper is to contribute to the mathematical analysis of this method. 

Let us formulate the class of penalty terms analyzed in this paper. 
We assume that $\Xspace$ is a space of functions on a bounded Lipschitz domain 
$\Omega\subset\mathbb{R}^d$ or the $d$-dimensional torus and consider 
two wavelet frames $\{\pframe_{j,k}\}_{(j,k)\in\Lambda}$ and 
$\{\dframe_{j,k}\}_{(j,k)\in\Lambda}$ in $L^2(\Omega)$, which are dual 
to each other, i.e.\ $\langle \pframe_{j',k'},\dframe_{j,k}\rangle_{L^2}=\delta_{jj'}\delta_{kk'}$. The index set is assumed to be of the form 
$\Lambda= \{(j,k) \colon j\in\mathbb{N}_0, k\in\Lambda_j\}$ where $\Lambda_j$ are 
finite sets such that 
$2^{jd}\leq |\Lambda_j|\leq C_\Lambda 2^{jd}$ for some constant $C_\Lambda\geq 1$.
%(Actually $\pframe_{0,k}$ will be scaling functions.)  
For parameters $s\in \mathbb{R}$, and $p,q\in [1,\infty]$ we define the norms 
\begin{align}\label{eq:w_besov_norm}
\|f\|_{s,p,q} := \left(\sum_{j=0}^\infty
2^{jqs}2^{jqd(\frac{1}{2}-\frac{1}{p})} 
\left(\sum_{k\in\Lambda_j} |\langle f, \dframe_{j,k}\rangle|^p\right)^{q/p}\right)^{1/q}
\end{align}
which are equivalent to the norm of the Besov space $B^s_{p,q}(\Omega)$ 
under certain conditions on the wavelet frame. 
In particular, the norm $\|\cdot\|_{0,1,1}$ is a weighted $l^1$-norm 
of the wavelet-coefficients.

In order to stably recover 
$f^{\dagger}$ from  noisy data $\gobs$ we consider Tikhonov regularization of the form 
%Besov space $\Xspace=B^0_{p,1}$ with $p\in [1,2]$ and $\Yspace$ an $L^2$ space. 
%A common approach to compute stable approximations of the true solution $f^\dagger$ given noisy data $\gobs$ is Tikhonov regularization of the form 
\begin{align}\label{Tyk}
\hat{f}_\alpha \in \argmin_{f\in D} \left[ \frac{1}{2} \|F(f)-\gobs\|_{\Yspace}^2 + \alpha \|f\|_{0,p,1} \right] 
\end{align}
with a regularization parameter $\alpha>0$ and $p\in [1,2]$. 
%Here $\|\cdot\|_{0,p,1}$ denotes a norm on the 
%Besov space $B^0_{p,1}$ in terms of the coefficients of a wavelet decomposition of $f$. 
%A main reason for the use of estimators of the form \eqref{Tyk} is that such penalty terms 
%promote sparsity, i.e.\ at most a finite number of wavelet coefficients of $\hat{f}_{\alpha}$ 
%not vanish (see e.g.~\cite{lorenz:08}).
We consider two noise models: In the standard deterministic noise model the observed data 
$\gobs$ satisfy
\begin{align}\label{deterministic_noise}
\gobs= g^\dagger+\xi \quad\text{with } \|\xi\|_{\Yspace} \leq \delta, 
\end{align} 
with the deterministic noise level $\delta\geq 0$. In our statistical noise model 
we assume that $\Yspace = L^2(\Omegaobs)$ with a bounded $d$-dimensional Riemannian 
manifold $\Omegaobs$, and data are given by 	
\begin{align}\label{statistical_noise}
 \gobs= g^\dagger+\epsilon Z 
\end{align}  
with statistical noise level $\epsilon\geq 0$ and some noise process $Z$ on $L^2(\Omegaobs)$ with white noise as prominent example. As white noise does not belong to $L^2(\Omegaobs)$ with probably $1$ we expand the square in \eqref{Tyk} and omit the term $\frac{1}{2}\|\gobs\|^2_{L^2}$, which has no influence on the minimizer. This yields 
\begin{align}\hat{f}_\alpha \in \argmin_{f\in D} \left[ \frac{1}{2} \|F(f)\|_{L^2}^2 -\langle \gobs, F(f)\rangle +\alpha  \|f\|_{0,p,1} \right].\label{Tykstat} \end{align} 
A main goal of regularization theory are bounds on the distance of 
regularized estimators $\hat{f}_\alpha$ 
to the true solution $f^{\dagger}$ in terms of the noise level $\delta$ or 
$\epsilon$, respectively. Let us review some relevant results on this topic 
in the literature. 
Tikhonov regularization for linear operators mapping from $l^1(\mathbb{N})$ 
to some Banach space $\Yspace$ has been studied in many papers, see \cite{DDD:04} for 
an influential early paper. 
For an injective operator with a continuous extension to $l^2(\mathbb{N})$ 
it has been shown in \cite{BL:08b,GHS:11} that sparsity of the true solution 
$f^{\dagger}$ is necessary and sufficient for linear convergence rates with respect to 
$\delta$. 
For signals $f^{\dagger}$ that are not necessarily sparse, sufficient conditions 
for convergence rates have been established in 
\cite{burger2013convergence,FHV:15}  in an abstract $l^1$-setting 
based on variational source conditions (see also \Cref{sec:comparison}). 
%However, these conditions are formulated in an abstract $l^1$ setting 
%without examples of practically relevant operators. 
Convergence rates with respect to the Bregman distance 
rather than the norm as loss function are shown in \cite{BHK:18} 
under Besov smoothness assumptions using approximate source conditions. 
In the special case that $F$ is linear and allows a wavelet-vaguelette 
decomposition involving $\{\pframe_{\lambda}\}$ and the norm in $\Yspace$ 
is chosen as $l^2$-norm of vaguelette coefficients, the estimators 
\eqref{Tyk} and \eqref{Tykstat} for $p=1$ coincide with wavelet shrinkage 
estimators (see \cite{donoho:1995,LL:01}). In this case optimal convergence 
rates under white noise with respect to the $L^2$-loss function have already been 
shown in these references. Optimal rates for finitely smoothing linear operators 
with white noise model are also established in \cite{KMR:06} for 
a two-step procedure consisting of wavelet shrinkage \emph{in the data space} 
$\Yspace$ followed by Tikhonov regularization.

The main results of this paper can be summarized as follows:
\begin{itemize}
	\item We formulate a \emph{general strategy for the 
verification of variational source conditions} for $l^1$-type penalties 
(\cref{vari:theorem}). (Recall that source conditions in regularization theory 
are conditions which allow to control the bias term/approximation error, 
see \cref{sec:vsc} for details.) 
This adapts a strategy based on a sequence of finite-dimensional subspaces 
and associated semi-norms which was proposed in \cite{HW:17,weidling2018optimal} 
for the case of smooth penalties. We show that our strategy is also a generalization 
of that in \cite{burger2013convergence}. In contrast to \cite{burger2013convergence}
we use conditions on the forward operator itself rather than its adjoint, which 
leads to some simplifications and allows the treatment of nonlinear operators. 
For linear operators our conditions are equivalent to those 
in \cite{burger2013convergence} (see \Cref{sec:comparison}). 
\item 
Using this strategy for the case of finitely smoothing operators and sufficiently 
smooth wavelet bases, we derive variational source conditions and by standard 
variational regularization theory also rates of convergence for \eqref{Tyk} in Besov scales with fixed integrability index $p\in[1,2]$ (\cref{rate:deterministic}). 
We also show the optimality of these rates. 
Whereas \emph{wavelet shrinkage} has been studied much more extensively, in particular 
for statistical noise models (see e.g.~\cite{donoho:1995,CHR:04}), 
to the best of our knowledge our results show for the first 
time that general sparsity promoting \emph{wavelet penalization} with smooth wavelets 
is capable of achieving optimal rates in all Besov spaces of a given integrability index. 
In particular, in contrast to quadratic Tikhonov regularization no saturation effect occurs, at least for infinitely smooth wavelets such as Meyer wavelets.  
With the regularization \eqref{Tyk} we can get arbitrarily close to a linear convergence rate without the rather restrictive assumption from \cite{BL:08b,GHS:11} 
that $f^{\dagger}$ be sparse in the wavelet basis. 
Note that the two-step procedure in \cite{KMR:06} also suffers from saturation. 
\item In \cref{theo:converse} we establish a \emph{converse result} for exact data. 
Recall that optimality as in the previous point means that no method can achieve a better 
\emph{uniform} rate over a certain Besov ball. Different Besov balls may have the same optimal rate. 
In contrast, converse results are about a single 
signal $f^{\dagger}$ and a fixed method, here \eqref{Tyk} 
(see \cite{KP:02} for a detailed discussion in a statistical context where 
sets characterized by converse results are called maxisets). 
I turns out that necessary conditions 
for power-type convergence rates in terms of $\alpha$ are described by Besov spaces with 
fine index $q=\infty$, and these conditions are also sufficient. Hence, such Besov spaces 
describe maximal sets (maxisets) on which a given rate is achieved. 
\item Finally, in \cref{theorem_stat_rates} 
we derive convergence rates for the white noise model \eqref{statistical_noise} 
using a method proposed in \cite{weidling2018optimal}. These rates are optimal 
for $p>1$ and almost optimal for $p=1$.
\end{itemize}

\section{Setting and Examples}
\subsection{assumptions}
We assume that $\Omega$ is either a bounded Lipschitz domain in $\mathbb{R}^d$ or 
the $d$-dimensional torus $(\mathbb{R}/\mathbb{Z})^d$. Moreover, suppose that  
each element of the frame $(\dframe_{\lambda})_{\lambda\in\Lambda}$ belongs 
to $C^{\overline{s}}(\overline{\Omega})$ for some $\overline{s}\in \mathbb{N}\cup \{\infty\} $. Let $\mathcal{D}^{\overline{s}}(\Omega)$ 
denote the space of distributions of order $\leq \overline{s}$ on $\mathbb{R}^d$ 
supported in $\overline{\Omega}$, or in case of the torus the space of periodic 
distributions of order $\leq \overline{s}$. Then the operators
$Q\colon \mathcal{D}^{\overline{s}}(\Omega)\rightarrow \mathbb{R}^{\Lambda}$ and 
$Q_j\colon \mathcal{D}^{\overline{s}}(\Omega)\rightarrow \mathbb{R}^{\Lambda_j}$
for $j\in\mathbb{N}_0$ are well-defined by 
\begin{align*}
Qf :=  \left(\langle f,\dframe_{j,k}\rangle\right)_{(j,k)\in\Lambda}
\qquad \text{and}\qquad 
Q_jf :=  \left(\langle f,\dframe_{j,k}\rangle\right)_{k\in\Lambda_j}.
\end{align*}
For $p,q\in [1,\infty]$ and $s\in \mathbb{R}$ we introduce the following spaces 
of sequences $(z_{\lambda})$ indexed by $\lambda\in\Lambda$:
\[
b^s_{p,q}:=\left\{(z_\lambda)\in\mathbb{R}^\Lambda \colon \|(z_\lambda)\|_{b^s_{p,q}}<\infty\right\}, \text{ }
\|(z_{\lambda})\|_{b^s_{p,q}}:= 
\left\|
\left(2^{js}2^{jd(\frac{1}{2}-\frac{1}{p})} 
\|(z_{j,k})_{k\in\Lambda_j}\|_p\right)_{j\in\mathbb{N}_0}\right\|_{l^q}
\]
Here $\|(z_{j,k})_{k\in\Lambda_j}\|_p:=\left(\sum_{k\in\Lambda_j}|z_{j,k}|^p\right)^{\frac{1}{p}}$ if $p\in [1,\infty)$ and 
$\|(z_{j,k})_{k\in\Lambda_j}\|_\infty:=\max_{k\in\Lambda_j} |z_{j,k}|$. 
Then the Besov-type norm in \eqref{eq:w_besov_norm} can be written as
\begin{align}\label{eq:normdef}
\|f\|_{s,p,q} = \|Qf\|_{b^s_{p,q}} = 
\left\| \left(2^{js}2^{jd(\frac{1}{2}-\frac{1}{p})} \|Q_j f\|_p\right)_{j\in \nat}\right\|_{\ell^q}
\end{align}
for all
$f\in \mathcal{D}^{\overline{s}}(\Omega)$, and we define 
\[
B^s_{p,q}:= B^s_{p,q}\left(\left\{\dframe_\lambda\right\}\right)
:= \left\{f\in \mathcal{D}^{\overline{s}}(\Omega)\colon 
\|f\|_{s,p,q}<\infty\right\}
\]
The following assumptions on these spaces will appear in the theorems below: 
\begin{subequations}\label{eq:setting}
\begin{align}\label{eq:completeness}
& B^s_{p,q} \text{ is complete for all } p,q\in [1,\infty] \text{ and }|s|<\overline{s},\\
\label{eq:Q_surjective}
&Q\colon B^s_{p,q} \to b^s_{p,q} \text{ is surjective for all }
p,q\in [1,\infty] \text{ and }|s|<\overline{s},\\
\label{eq:p_in_12}
 & p\in [1,2] \text{ and the domain } D \text{ of } F \text{ is a non-empty convex subset of } B^0_{p,1},\\
\label{eq:domain_closed}
& a\in ( d/p-d/2  , \overline{s})
\text{ and } D \text{ is closed in the topology induced by } \|\cdot \|_{-a,2,1} \text{ on } B^0_{p,1}.
\end{align}
\end{subequations} 
Note that due to the completeness of $b^s_{p,q}$, assumption \eqref{eq:Q_surjective} 
is stronger than \eqref{eq:completeness}. Further note that if 
$\{\dframe_{\lambda}\colon\lambda\in\Lambda\}$ is Riesz basis of $L^2(\Omega)$ 
(which will be assumed for lower bounds in \Cref{sec:lowerbounds}), then 
\eqref{eq:Q_surjective} holds true for $s=0$ and $p=q=2$, and $B^0_{2,2}=L^2(\Omega)$. 
\eqref{eq:Q_surjective} for the non-separable spaces with $\max(p,q)=\infty$ is 
a stronger assumption and will only be needed in \Cref{sec:sparsity}. The lower bound on the number $a$ in \eqref{eq:domain_closed} implies  $B^0_{p,1}\subset B^{-a}_{2,1}$ (see \ref{prop:embeddings}\ref{embed_mixed}). As the embedding is compact 
(see \Cref{compactness_embedding}), we cannot 
hope that $D$ is complete in the $B^{-a}_{2,1}$ topology. If $a\in ( d/p-d/2  , \overline{s})$ then \eqref{eq:domain_closed} is satisfied if and only if $D$ is the intersection of a closed set in $B^{-a}_{2,2}$ with $B^0_{p,1}$. For example 
box-constrained sets $D=\{f\in B^0_{p,1}:\alpha\leq f\leq \beta\;\text{a.e.}\}$ 
with $\alpha,\beta\in\mathbb{R}$, $\alpha<\beta$ satisfy 
\eqref{eq:domain_closed} if $B^0_{p,1}(\{\dframe_\lambda\}) = B^0_{p,1}(\Omega)$ and 
$B^{-a}_{2,1}(\{\dframe_\lambda\}) = B^{-a}_{2,1}(\Omega)$.

Most of our  theorems also need some of the following assumption on the operator 
$F$ which are parameterized by some number $a>0$.
%Let $B^s_{p,q}$ be the completion of 
%$\left\{ f\in L^2(\Omega) \colon \|f\|_{s,p,q} < \infty \right\}$ with respect to $\|\cdot\|_{s,p,q}$. By the frame inequalities we have $B^0_{2,2}=L^2(\Omega)$ and the norms $\|\cdot \|_{0,2,2}$ and $\|\cdot\|_{L^2}$ are equivalent. The linear map $Q_n$ extends to $B^s_{p,q}$ and \eqref{eq:normdef} remains valid for all $f\in B^s_{p,q}$.\\
%\begin{proposition} 
%\begin{enumerate}
%\item Let $p,q\in[1,\infty)$, $s\in \mathbb{R}$. Let $p^\prime,q^\prime$ be the Hölder conjugates of $p$ respectively $q$. 
%The bilinear map 
%\[ \langle\cdot ,\cdot \rangle \colon B^{-s}_{p^\prime,q^\prime } \times B^s_{p,q}  \rightarrow \mathbb{R}\quad \text{given by } (g,f)\mapsto  \sum_{j=0}^\infty \langle Q_j g, Q_j f\rangle\] 
%satisfies $\langle g,f \rangle \leq \|g\|_{-s,p^\prime,q^\prime}\| \|f\|_{s,p,q}$.
%\item Let $p\in [1,2]$ and  $a \geq d\left(\frac{1}{p}- \frac{1}{2}\right)$. Then 
%$ \|f\|_{-a,2,2}\leq \|f\|_{0,p,1}$ for all $f\in B^0_{p,1}$. Hence $B^0_{p,1}$ is continuously embedded in $B^{-a}_{2,2}$.
%If $a > d\left(\frac{1}{p}- \frac{1}{2}\right)$, then the embedding $B^0_{p,1}\subseteq B^{-a}_{2,2}$ is compact.
%\item Let $p,q_1,q_2 \in[1,\infty)$ with $q_1\leq q_2$ and $s\in \mathbb{R}$. Then $\|f\|_{s,p,q_2}\leq \|f\|_{s,p,q_1}$ for all $f\in B^s_{p,q_1}$. Hence $B^{s}_{p,q_1}$ is continuously embedded in $B^{s}_{p,q_2}$.
%\item  Let  $s>0$. Then $\|f\|_{0,p,1}\leq (1-2^{-s})^{-1} \|f\|_{s,p,\infty}$  for all $f\in B^s_{p,\infty}.$ Hence there exists a continuous embedding $B^s_{p,\infty}\subseteq B^{0}_{p,1}$.
%
%\end{proposition} 
\begin{assumption}\label{Finvlipschitz} 
There exists a constant $L_1>0$ such that 
\begin{align}\label{eq:Finvlipschitz}
 \|f_1-f_2\|_{-a,p,\infty} & \leq L_1 \|F(f_1)-F(f_2)\|_{\Yspace} \quad \text{for all } f_1, f_2\in D.  
  %\text{and} 
%\|F(f_1)-F(f_2)\|_{\Yspace}&  \leq L_d \|f_1-f_2 \|_{0,p,1 }\quad \text{for all } f_1, f_2\in D. 
\end{align} 
\end{assumption} 

\begin{assumption}\label{Flipschitz} 
There exists a constant $L_2>0$ such that 
\begin{align*}
\|F(f_1)-F(f_2)\|_{\Yspace}\leq L_2 \|f_1-f_2\|_{-a,2,1}\quad\text{for all } f_1, f_2\in D.
\end{align*}

%There exists a constant $L_D$ and a Banach space $\Xspace_D$ with a compact embedding $B^0_{p,1}\subset \Xspace_D$ such that the domain $D$ of the operator $F\colon D\subseteq B^0_{p,1}\rightarrow \Yspace$ 
%is non-empty and closed in $\Xspace_D$ with 
%\begin{align}\label{eqq:FLipschitz}
%\|F(f_1)-F(f_2)\|_{\Yspace} \leq L_D \|f_1-f_2\|_{\Xspace_D}  
%\end{align}
%Furthermore, assume, that $\|\cdot\|_{0,p,1}$ is lower semicontinuous on $B^0_{p,1}$ with respect to the topology induced by $\|\cdot \|_{\Xspace_D}$.
\end{assumption}

\begin{assumption}\label{operator_additional}
Assume $\Yspace = L^2(\Omegaobs)$ with a bounded $d$-dimensional Riemannian 
manifold $\Omegaobs$ and that there exists a constant $L_3>0$ such that 
\begin{align*} 
 \left\|F(f_1)-F(f_2)\right\|_{B^a_{p,1}(\Omegaobs)} \leq L_3\|f_1-f_2\|_{0,p,1} \quad\text{for all } f_1,f_2\in D. 
\end{align*} 
\end{assumption}

\begin{assumption}\label{operator_s}
Assume $\Yspace = L^2(\Omegaobs)$ with a bounded $d$-dimensional Riemannian 
manifold $\Omegaobs$ and that for some constants $L_4,s>0$ we have 
\[
\|f_1-f_2\|_{s,p,\infty}\leq L_4\|F(f_1)-F(f_2)\|_{B^{s+a}_{p,\infty}(\Omegaobs)}
\quad \text{for all }f_1,f_2\in D\cap B^s_{p,\infty}. 
\]
\end{assumption}

Usually $\Yspace$ is chosen as an $L^{\tilde{p}}$ space, most often with  
$\tilde{p}=2$, and this case will be considered in all our examples below. 
But we keep our assumptions as general as possible to cover also other interesting 
choices of the data fidelity term,  e.g.\ $\tilde{p}=1$ for impulsive noise. 
If $a\geq d/p-d/2$ then Assumptions \ref{Finvlipschitz} and \ref{Flipschitz} follow from the stronger conditions 
\begin{align}\label{eq:a_times_smoothing}
\frac{1}{L_1'}\|f_1-f_2\|_{-a,2,2}\leq 
\|F(f_1)-F(f_2)\|_{\Yspace}\leq L_2' \|f_1-f_2\|_{-a,2,2}
\qquad \mbox{for all }f_1,f_2\in D
\end{align}
(see \Cref{prop:embeddings}),  
and these inequalities will be verified for all our examples. 
\Cref{Finvlipschitz} or the first inequality in \eqref{eq:a_times_smoothing}
imposes that 
$F$ is \emph{at most $a$ times smoothing}.
This will be required for upper bounds on the reconstruction error. 
On the other hand, Assumptions \ref{Flipschitz}, \ref{operator_additional} 
and the second inequality in \eqref{eq:a_times_smoothing}
state  that $F$ is \emph{at least $a$ times 
smoothing} and will be used mostly to derive lower bounds on the reconstruction error. 
Assumptions \ref{operator_additional} and \ref{operator_s} are only needed in the stochastic case. 

Finally, we point out that an operator $F$ satisfying \eqref{eq:a_times_smoothing} 
and defined initially only on $D\cap L^2(\Omega)$ has a unique continuous extension 
to $D\subset B^0_{p,1}$ if $a\geq d/p-d/2$, \eqref{eq:p_in_12} and $D\cap L^2(\Omega)$ is dense in $D$ with respect to $\|\cdot\|_{0,p,1}$. 

\subsection{examples}
We now sketch the verification of our assumptions in several examples.

\begin{example}[periodic wavelets]
\emph{wavelets on $\mathbb{R}$:}
Let $\pframe^{\mathrm{M}}\in C^{\overline{s}}(\mathbb{R})$ be a (mother) wavelet 
with corresponding father wavelet or $\pframe^{\mathrm{F}}\in C^{\overline{s}}(\mathbb{R})$. 
We assume for simplicity that the system  given by 
$\pframe_{0,k}^0(x):=\pframe^{\mathrm{F}}(x-k)$ and 
$\pframe_{j+1,k}^0(x):= 2^{j/2}\pframe^{\mathrm{M}}(2^jx-k)$, $j\in \mathbb{N}_0$, 
forms an orthonormal basis 
$\{\pframe_{j,k}^0\colon j\in \mathbb{N}_0,k\in\mathbb{Z}\}$ 
of $L^2(\mathbb{R})$. We further assume that this wavelet system is 
$\overline{s}$-regular, $\overline{s}>0$ in the sense of 
\cite[Def.~4.2.14]{GN:15}, which 
entails in particular that $\int_{\mathbb{R}}\pframe^{\mathrm{M}}(x)x^m\,dx=0$ for 
all $m\in \mathbb{N}_0$  with $m\leq\overline{s}-1$. Examples of such wavelets include 
Daubechies wavelets \cite{daubechies:88} of order $N\in\mathbb{N}$, which 
are supported in an interval of length $2N-1$ and are $\overline{s}=0.193(N-1)$-regular, as well as Meyer-wavelets \cite{meyer:92}, 
which are $\overline{s}$-regular for any $\overline{s}>0$. 

\emph{periodization:} We define the periodized functions 
$\pframe^{\mathrm{P},G}_j(x):= \sum_{k\in\mathbb{Z}} 2^{j/2} 
\pframe^{G}(2^j(x-k))$ for $G\in\{\mathrm{F,M}\}$ and $j\in \mathbb{N}_0$ 
and set 
$\pframe^{\mathrm{P}}_{0,0} := \pframe^{\mathrm{F}}_{0}$ and  
$\pframe^{\mathrm{P}}_{j+1,k}(x): = \pframe^{\mathrm{P,M}}_j(x-2^{-j}k)$ 
for $j\in \mathbb{N}_0$. Note that all these functions belong to the space 
$C^{\overline{s}}(\mathbb{R}/\mathbb{Z})$ of $1$-periodic functions in 
$C^{\overline{s}}(\mathbb{R})$. 
It is easy to see that $\{\pframe^\mathrm{P}_{j,k}\colon j,k\in \Lambda^\mathrm{P}\}$ 
with $\Lambda^{\mathrm P}:=\{(j,k)\in \mathbb{N}_0^2\colon 
k<2^{j-1}\}$ defines an orthonormal system of $L^2(\mathbb{R}/\mathbb{Z})$. 
Moreover, setting $\dframe_{\lambda}^{\mathrm P}:=\pframe_{\lambda}^\mathrm{P}$
for $\lambda\in\Lambda^{\mathrm P}$, we have 
\begin{align*}
B^s_{p,q}\left(\left\{\dframe_\lambda^{\mathrm P}\right\}\right) = B^s_{p,q}(\mathbb{R}/\mathbb{Z})\qquad 
\mbox{for }|s|<\overline{s}
\end{align*}
with equivalent norms, and \eqref{eq:completeness}, \eqref{eq:Q_surjective} hold true; 
see \cite[Thm.~4.3.26 and (4.137)]{GN:15}.

\emph{tensor product wavelets on $(\mathbb{R}/\mathbb{Z})^d$:}
Let $\mathcal{G}:= \{\mathrm{F},\mathrm{M}\}^d\setminus\{(\mathrm{F},\dots,\mathrm{F})\}$ and note that $|\mathcal{G}| = 2^d-1$. 
We set $\pframe_{0,\mathbf{0}}(\mathbf{x}):=\prod_{l=1}^d \pframe^{\mathrm{P,F}}_0(x_l)$ and 
$\pframe_{j,\mathbf{k},\mathbf{G}}:= \prod_{l=1}^d \pframe^{\mathrm{P},G_l}_j(x_l-2^{-j}k_l)$ for 
$\mathbf{k} \in \{0,\dots,2^j-1\}^d$, $\mathbf{G}\in \mathcal{G}$, and $j\in\mathbb{N}$. 
Set $\Lambda_0^{\mathrm{P},d}:=\{\mathbf{0}\}$, 
$\Lambda_{j+1}^{\mathrm{P},d}:= \{0,\dots,2^j-1\}^d\times \mathcal{G}$ for $j\in\mathbb{N}_0$ and 
$\Lambda^{\mathrm{P},d}:=\{(j,l)\colon j\in\mathbb{N}_0, l\in\Lambda_j^{\mathrm{P},d}\}$. 
Then $\{\pframe_{\lambda}^{\mathrm{P},d}\colon \lambda\in\Lambda^{\mathrm{P},d}\}$ 
is an orthonormal basis of $L^2((\mathbb{R}/\mathbb{Z})^d)$, and we set 
$\dframe_{\lambda}^{\mathrm{P},d}:=\pframe_{\lambda}^{\mathrm{P},d}$. 
Moreover, \eqref{eq:completeness} and \eqref{eq:Q_surjective} hold true and 
\begin{align*}
B^s_{p,q}\left(\left\{\dframe_\lambda^{\mathrm{P},d}\right\}\right) = B^s_{p,q}\left((\mathbb{R}/\mathbb{Z})^d\right)\qquad 
\mbox{for }|s|<\overline{s}.
\end{align*}

\emph{$a$-times smoothing operators:}
Examples of linear operators which satisfy Assumptions \ref{Finvlipschitz}, 
\ref{Flipschitz}, \ref{operator_additional}, and \ref{operator_s} with 
$\Yspace = L^2((\mathbb{R}/\mathbb{Z})^d)$ are inverses of elliptic 
differential operators of order $a$ with smooth, periodic coefficients 
(see \cite{Triebel1978}). 
Other examples are periodic convolution operators $Ff := k*f$ for which 
the Fourier coefficients of the convolution kernels $k$ have the asymptotic 
behavior $\widehat{k}(\mathbf{n})\sim (1+|\mathbf{n}|^2)^{-a/2}$, 
$\mathbf{n}\in\mathbb{Z}^d$.
\end{example}

Let us recall two possibilities to define Besov and Sobolev spaces on open 
subdomains $\Omega\subset \mathbb{R}^d$: We denote by $B^s_{p,q}(\Omega)$ the 
set of all restrictions $f|_{\Omega}$ (in the sense of the theory of 
distributions) of $f\in B^s_{p,q}(\mathbb{R}^d)$. 
Moreover, let $\mathring{B}^s_{p,q}(\Omega)$ be the closure of all functions 
$f\in C^{\infty}(\mathbb{R}^d)$ with $\supp f\subset \Omega$ in  
$B^s_{p,q}(\mathbb{R}^d)$.

\begin{example}[wavelets on bounded subsets of $\mathbb{R}^d$]
Let $\Omega \subset \mathbb{R}^d$ be a bounded Lipschitz domain. 
Then it is possible under certain conditions to modify the wavelets 
whose support intersects with the boundary $\partial \Omega$ (or is close to 
$\partial\Omega$ relative to its size) such that approximation properties and inverse 
inequalities of the corresponding wavelet spaces are preserved. This was done 
in \cite{DKU:99} for the symmetric, compactly supported biorthogonal wavelets 
from \cite{CDF:92}. There are different constructions of boundary-adapted wavelets 
which satisfy \eqref{eq:completeness} and \eqref{eq:Q_surjective} and yield either 
\begin{align}\label{eq:Besov_domain0}
& B^s_{p,q}\left(\{\dframe_\lambda\}\right)=\mathring{B}^s_{p,q}(\Omega) \qquad \mbox{or}\\
\label{eq:Besov_domain}
& B^s_{p,q}\left(\{\dframe_\lambda\}\right)= B^s_{p,q}(\Omega)
\end{align}
for certain values of $s,p$ and $q$ (see also \cite[\S 3.9--3.10]{cohen:03}, 
\cite[\S 4.3.5]{GN:15}, and \cite{triebel:08}). 
\end{example}

\begin{example}[Radon transform]\label{ex:radon}
We consider the Radon transform 
on a bounded domain $\Omega\subset \mathbb{R}^d$, which appears 
as forward operator $F$ in computed tomography (CT) and  positron emission tomography 
(PET), among others. If $S^{d-1}:=\{x\in\mathbb{R}^d:|x|=1\}$ is the unit sphere, 
$R:B^0_{p,1}(\Omega)\to L^2(S^{d-1}\times \mathbb{R})$ is given by
\[
(Rf)(\theta,t) := \int_{x\cdot \theta = t} f(x)\,dx,\qquad \theta\in S^{d-1},
t\in\mathbb{R}.
\]
As a special case of \cite[Thm.~5.1]{natterer:86} the Radon transform 
satisfies the inequality 
\[
\frac{1}{L} \|f\|_{B^{(1-d)/2}_{2,2}(\Omega)} 
\leq \|Rf\|_{L^2(S^{d-1}\times \mathbb{R})} 
\leq L \|f\|_{B^{(1-d)/2}_{2,2}(\Omega)}
\qquad \mbox{for all }f\in L^2(\Omega).
\]
%Together with the continuous embeddings 
%\begin{equation}\label{eq:besov_embeddings_q}
%B^{-a}_{2,1}(\Omega) \subset B^{-a}_{2,2}(\Omega) \subset 
%B^{-a}_{2,\infty}(\Omega) 
%\end{equation}
%this shows that 
Together with the density of $L^2(\Omega)$ in $B^{(1-d)/2}_{2,2}(\Omega)$ 
this shows that \eqref{eq:a_times_smoothing} and hence 
Assumptions \ref{Finvlipschitz} and \ref{Flipschitz} are satisfied with 
$a=\frac{d-1}{2}$ and $\Yspace = L^2(S^{d-1}\times \mathbb{R})$ 
if the wavelet system is chosen such that \eqref{eq:Besov_domain} holds true 
for $s=-\frac{d-1}{2}$ and $p=2$.  Moreover, it follows from \cite[Thm.~3.1]{hertle:83} 
that $R$ is an isomorphism from $B^s_{2,2}(\{x\in\mathbb{R}^d:x\leq 1\})$ 
to $B^{s+a}_{2,2}(S^{d-1}\times [-1,1])$ for all $s\in\mathbb{R}$, and hence 
Assumptions \ref{operator_additional} and \ref{operator_s} with $p=2$ are 
satisfied as well by K-interpolation theory (see \cite{Triebel1978}). 
\end{example}

\begin{example}[nonlinear operators]
In the study of nonlinear operator equations in Hilbert scales the following 
condition, which is closely related to Assumptions \ref{Finvlipschitz} and \ref{Flipschitz}, has been investigated (\cite{HP:08,neubauer:92}):
\begin{align}\label{eq:Frechet_cond}
\frac{1}{\tilde{L}}\|h\|_{\Xspace_{-a}}
\leq \|F'[f^{\dagger}]h\|_{\Yspace} \leq \tilde{L}\|h\|_{\Xspace_{-a}}.
\end{align}
Here $\tilde{L}>0$, and $\Xspace_{-a}$ is an element of a Hilbert scale, 
typically of $L^2$-based Sobolev spaces on $\Omega$ with smoothness index $-a$. 
We also need the so-called \emph{range invariance condition}: 
For all $f_1,f_2\in D$ there exists some operator 
$R(f_1,f_2)\in L(\Yspace)$ such that 
\begin{align}\label{eq:range_invariance}
F'[f_1] = R(f_1,f_2)F'[f_2]\qquad \mbox{and}\qquad \|I-R(f_1,f_2)\|\leq \frac{1}{2}.
\end{align}
(Often a bound  $\|I-R(f_1,f_2)\|\leq C\|f_1-f_2\|$ is 
shown, which implies \eqref{eq:range_invariance} in a ball of radius 
$1/(2C)$.) 

The following lemma shows that \eqref{eq:Frechet_cond} and \eqref{eq:range_invariance}
implies \eqref{eq:a_times_smoothing} and hence 
Assumptions \ref{Finvlipschitz} and \ref{Flipschitz} if 
\(\Xspace_{-a} = B^{-a}_{2,2}\) % \(\Xspace_{-a} = B^{-a}_{2,2}(\{\dframe_{\lambda}\})\)
with equivalent norms. 
In particular, the conditions \eqref{eq:Frechet_cond} and 
\eqref{eq:range_invariance} have been verified for the following nonlinear 
inverse problems: 
\begin{itemize}
\item \emph{identification of a reaction coefficient $c$.} 
Let $\Omega\subset\mathbb{R}^3$, $d\in\{1,2,3\}$ be a bounded 
Lipschitz domain and $f$ and $g$ smooth right hand sides. 
  For a given $c\in L^{\infty}(\Omega)$ satisfying $c(x)\geq \underline{c}>0$ 
almost everywhere, we define $F(c):=u$ where $u$ solves the elliptic 
boundary value problem  
\begin{align}
\begin{aligned}
&-\Delta u + c u = f&&\mbox{in }\Omega,\\
&u=g&&\mbox{on }\partial \Omega.
\end{aligned}
\end{align}
For this problem eq.~\eqref{eq:Frechet_cond} with $\Yspace = L^2(\Omega)$ 
and $a=2$ has been shown in 
\cite[Thm.~4.5]{HP:08}, and eq.~\eqref{eq:range_invariance}
in \cite[Ex.~4.2]{HNS:95}.
\item \emph{identification of a diffusion coefficient $\rho$.} 
Given $\rho\in L^\infty([0,1])$ with $\rho\geq \underline{\rho}>0$, 
we define $F(\rho):=u$ where $u$ solves the boundary value problem 
\begin{align}
\begin{aligned}
&-(\rho u')' = f&&\mbox{in }(0,1),\\
&u(0)=g_0,\quad u(1)=g_1.
\end{aligned}
\end{align}
Here \eqref{eq:Frechet_cond} with $a=1$ and $\Yspace = L^2([0,1])$ 
has been verified in  \cite[Thm.~5.4]{HP:08} (in a Hilbert scale of 
Sobolev spaces with mean $0$ that was shifted for technical reasons). 
Moreover,  \eqref{eq:range_invariance} was shown in \cite[Ex.~4.3]{HNS:95}.
\item \emph{Hammerstein integral equations.} The forward operator 
is defined by 
\[
(F(f))(t) := \int_0^t\phi(f(s))\,ds
\] 
where $\phi\in C^{2,1}(I)$ on all intervals $I\subset\mathbb{R}$. 
In this case \eqref{eq:Frechet_cond} with $a=1$ and $\Yspace = L^2([0,1])$ 
is shown in  \cite[\S 4]{neubauer:00}, 
and \eqref{eq:range_invariance} in \cite[Ex.~4.1]{HNS:95}.
\end{itemize}
The verification of Assumptions \ref{operator_additional} and \ref{operator_s}
needed for statistical noise models seems to be less straightforward for these 
examples and is left for future research. 
\end{example}

\begin{lemma}
Let $\Xspace_{-a}$ and $\Yspace$ be Banach spaces and  
suppose that $D\subset \Xspace_{-a}$ is convex and $F:D\subset\Xspace_{-a}\to \Yspace$ 
is Fr\'echet differentiable and satisfies \eqref{eq:Frechet_cond} and 
\eqref{eq:range_invariance}. Then 
\[
\frac{1}{2\tilde{L}}\|f_1-f_2\|_{\Xspace_{-a}} 
\leq \|F(f_1)-F(f_2)\|_{\Yspace}
\leq \frac{3}{2}\tilde{L}\|f_1-f_2\|_{\Xspace_{-a}}, \qquad f_1,f_2\in D. 
\]
\end{lemma}

\begin{proof}
By the mean value theorem and eq.~\eqref{eq:range_invariance} we have 
\begin{align*}
F(f_1)-F(f_2) &= \int_0^1 F'[f_2+t(f_1-f_2)](f_1-f_2)\,dt \\
 &= \left(\int_0^1 \left(R(f_2+t(f_1-f_2),f^{\dagger}) -I\right)\,dt+ I\right)
F'[f^{\dagger}](f_1-f_2).
\end{align*}
As $\|\int_0^1 (R(f_2+t(f_1-f_2),f^{\dagger})-I)\,dt\|\leq \frac{1}{2}$, we obtain
\[
\frac{1}{2}\|F'[f^{\dagger}](f_1-f_2)\|_{\Yspace}
\leq \|F(f_1)-F(f_2)\|_{\Yspace} 
\leq \frac{3}{2} \|F'[f^{\dagger}](f_1-f_2)\|_{\Yspace}.
\]
Together with \eqref{eq:Frechet_cond} this yields the assertion. 
\end{proof}

\section{Variational source conditions}\label{sec:vsc}
In subsection \ref{sec:verification} and in the following brief review of 
variational source condition we consider a more general setting than in the 
rest of this paper: Let $\Xspace$, $\Yspace$ be Banach spaces.
Suppose $F \colon D \subset \Xspace \rightarrow \Yspace$ is an injective, continuous map with $D \subset \Xspace$ non-empty and discontinuous inverse. Let $f^\dagger\in D$ be the true solution to an equation $F(f^\dagger)=g^\dagger$ for given exact data $g^\dagger\in \Yspace$. 

Variational source conditions are sufficient and often even necessary conditions 
for rates of convergence of Tikhonov regularization and other regularization methods 
(\cite{Scherzer_etal:09,flemming:12b,HW:17}). 
In general such conditions have the following form: Let
$\loss:\Xspace\times \Xspace\to [0,\infty)$ with $\loss(f,f)=0$ for all $f\in \Xspace$ 
be a loss function  to measure reconstruction errors and let 
$\varphi\colon [0,\infty)\rightarrow [0,\infty)$ be continuous and strictly increasing 
with $\varphi(0)=0$. (Such functions are called index functions.)   
We say that the true solution $f^{\dagger}$ satisfies a variational source 
condition with index function $\varphi$ with respect to the loss function $l$ if  
\begin{align}\label{eq:varicond}
\loss\left(f,f^\dagger\right) \leq \|f\|_\Xspace-\|f^\dagger \|_\Xspace 
+\varphi\left(\|F(f)- F(f^\dagger)\|_\Yspace^2\right) \quad\text{ for all }f\in D. 
\end{align}
It can be shown (see \Cref{prop:standard_conv}) that for the deterministic noise model 
\eqref{deterministic_noise} and a proper choice of the regularization parameter 
$\alpha$ in Tikhonov regularization with penalty term $\alpha\|f\|_{\Xspace}$, 
the source condition \eqref{eq:varicond} with concave $\varphi$ 
implies the convergence rate 
\[
\loss\left(\hat{f}_{\alpha(\delta,\gobs)},f^{\dagger}\right) 
\leq 2\varphi(\delta^2). 
\]
\eqref{eq:varicond} will also lead to optimal statistical convergence rates. 
%In \Cref{convergence} we will see, that such a variational source condition implies convergence rates for both noise models. \\
%\begin{proposition}[{\cite[Proposition 13]{flemming2018injectivity}}]\label{vari_implies_rates}
%Let the variational source condition \eqref{eq:varicond} be satisfied. If $\alpha$ in \eqref{Tyk} is chosen such that $\alpha\sim \frac{\delta^2}{\varphi(\delta)}$, then there exists a constant $C>0$ such that 
%\[ \|\hat{f}_\alpha - f^\dagger\|_\Xspace \leq C \varphi(\delta) \quad\text{for all } \delta>0.\]
%\end{proposition} 
In the first paper on variational source conditions \cite{HKPS:07} as well as 
in many of the subsequent publications the loss function is chosen as a scalar 
multiple of the Bregman divergence. (For a discussion of the relation of 
\eqref{eq:varicond} to other types of source conditions, e.g.~classical 
spectral source conditions we refer to \cite{flemming:12b}.) 
However, even if the Bregman divergence 
of the $l^1$-norm vanishes, the norm difference of the two elements can be 
arbitrary large. In this sense error bounds with respect to the Bregman divergence 
are not very informative for $l^1$ regularization, and therefore we will use the 
norm in the penalty term to define the loss function instead.  

\subsection{verification of variational source 
condition for  \texorpdfstring{$1$}{}-homogeneous penalties} \label{sec:verification}
In \cite[Thm.~2.1]{HW:17} and \cite[Thm.~3.3]{weidling2018optimal} a general 
strategy for the verification of variational source conditions with Bregman loss 
has been proposed. The following theorem provides an analog for $l^1$-type 
loss functions and $l^1$-type penalties. An essential prerequisite is the 
Pythagoras type equality in \eqref{eq:strategy_additive} with exponent $1$.
\begin{proposition}\label{vari:theorem}
Suppose there exists a family of seminorms $(p_n\colon \Xspace\rightarrow [0,\infty))_{n\in\nat}$ and an increasing sequence $(\nu_n)_{n\in\nat}$ of positive real numbers such that the following conditions are satisfied: 
\begin{subequations} \label{eq:strategy}
\begin{align}
&  p_n^\perp:= \|\cdot\|_\Xspace- p_n  \quad\text{is a seminorm on } \Xspace \text{ for all } n\in\nat,\label{eq:strategy_additive}  \\ 
& \lim_n p_n (f^\dagger)=\|f^\dagger\|_\Xspace, \label{eq:strategy_limit}\\
&  p_n(f_1-f_2)\leq \nu_n \|F(f_1)-F(f_2)\|_\Yspace \quad \text{for all } f_1,f_2 \in D, n\in\nat.\label{eq:strategy_operator}
\end{align}
\end{subequations} 
Then \eqref{eq:varicond} holds true with $\loss(f,f^{\dagger})= \|f-f^{\dagger}\|_\Xspace$ and 
\[ \varphi(t)= 2 \inf_{n\in \mathbb{N}} \left( \nu_n \sqrt{t} + p_n^\perp( f^\dagger) \right).\]
\end{proposition}
\begin{proof}
First we prove that $\varphi$ is indeed an index function. The conditions \eqref{eq:strategy_additive} and \eqref{eq:strategy_limit} imply  $\varphi(0)=0$. 
As it is an infimum of upper semicontinuous, concave functions $\varphi$ is upper semicontinuous and concave. Concavity yields continuity on $(0,\infty)$. Using upper semicontinuity and $\varphi(0)=0$, we obtain continuity at $0$. Let $t_1<t_2$. The discontinuity of the inverse of $F$ and \eqref{eq:strategy_limit} yield $\nu_n\rightarrow \infty$. Therefore, there exists $n_0\in\mathbb{N}$ such that the infinmum involved in the definition of $\varphi(t_2)$ is attained at $n_0$. We obtain
\[ \varphi(t_1)\leq  2 \left( \nu_{n_0} \sqrt{t_1} + ( f^\dagger)\right)  <  2 \left( \nu_{n_0} \sqrt{t_2} + p^\perp_{n_0})( f^\dagger)\right)  =\varphi(t_2).\]
This proves that $\varphi$ is strictly increasing.\\
Next we turn to the proof of \eqref{eq:varicond} (see also \cite[Lemma 5.1, Theorem 5.2]{burger2013convergence}). Let $f\in \Xspace$ and $n\in\mathbb{N}$. Using  the triangle equality for the $p_n^\perp$ terms and the reverse triangle inequality for the $p_n$ terms, we achieve 
\begin{align*}
\|f-f^\dagger\|_\Xspace-\|f\|_\Xspace+ \|f^\dagger\|_\Xspace  & = p_n(f-f^\dagger)-p_n(f)+p_n(f^\dagger) \\
& +p_n^\perp(f-f^\dagger)-p_n^\perp(f)+p_n^\perp(f^\dagger)\\
& \leq 2 \left(  p_n(f-f^\dagger) +  p_n^\perp(f^\dagger)\right)\\
& \leq 2 \left( \nu_n \|F(f)- F(f^\dagger) \|_\Yspace+ p_n^\perp(f^\dagger)\right)
\end{align*}
for all $n\in\mathbb{N}$. 
Taking the infimum over $n$ on the right hand side proves the assertion. 
\end{proof}
\begin{example}\label{ell1example} Consider the Banach space $\ell^1$ of real absolute summable sequences. For $n\in\mathbb{N}$ we define $p_n\colon \ell^1\rightarrow [0,\infty)$ by $p_n(x)= \sum_{k=1}^n |x_k|$. Then \eqref{eq:strategy_additive} holds true and \eqref{eq:strategy_limit} is satisfied for all $x^\dagger\in\ell^1.$ 
In the \cref{sec:comparison} we prove that for a bounded linear operator $\linOp \colon \ell^1 \rightarrow \Yspace$ to some Banach space $\Yspace$ condition \eqref{eq:strategy_operator} is equivalent to $e_n\in \mathcal{R}(\linOp^\ast)$ for all $n\in\mathbb{N}$. Here $e_n = (\delta_{n,m})_{m\in\mathbb{N}}\in \ell^\infty$ denotes the $n$th standard unit sequence. In \cite{burger2013convergence} convergence rates for $\ell^1$-regularization are proven under this condition on $\mathcal{R}(\linOp^\ast)$.
\end{example}

\subsection{verification for finitely smoothing operators}\label{verification for finitely smoothing operators}

We will apply \Cref{vari:theorem} to an $a$-times smoothing operator in the 
sense of \Cref{Finvlipschitz} to obtain a variational source condition with 
loss function given by $\|\cdot\|_{0,p,1}$. For $n\in\nat$ we define a seminorm 
\[ p_n\colon B^0_{p,1}\rightarrow [0,\infty) \quad\text{by}\quad p_n(f):=\sum_{j=0}^n 2^{jd(\frac{1}{2}-\frac{1}{p})} \|Q_jf \|_p.\] 
%As the functionals
%$p_n^\perp(f):=\|f\|_{0,p,1}-p_n(f)$ are given by 
% \[p_n^\perp(f)=\sum_{j=n+1}^\infty 2^{jd(\frac{1}{2}-\frac{1}{p})} \|Q_jf \|_p,\] 
%they are seminorms, i.e.\ assumption \eqref{eq:strategy_additive} is satisfied.  
%If $f\in B^0_{p,1}$, then $\lim_n p_n(f)=\|f\|_{0,p,1}$. Hence \eqref{eq:strategy_limit} is always satisfied. 	 
\begin{lemma} \label{lemma:equivalentassumptions}
Suppose $F\colon D\subseteq B^0_{p,1}\rightarrow \Yspace$ is a map with a non-empty domain $D$ and $a>0$. Then the following statements are equivalent:
\begin{enumerate}
\item There exists $L_1>0$ such that \eqref{eq:Finvlipschitz} is satisfied. 
\item There exists $L_\nu>0$ such that \eqref{eq:strategy_operator} holds with $\nu_n = L_\nu 2^{na}$.
\end{enumerate}
More precisely: If \eqref{eq:Finvlipschitz} is satisfied with $L_1>0$ then  \eqref{eq:strategy_operator} holds with \[ \nu_n=L_1 (2^a-1)^{-1}  2^{a(n+1)} .\] The second statement implies the first with $L_1:=L_\nu$. 
\end{lemma} 
\begin{proof} 
Let $f_1,f_2\in D$ and $n\in \nat$.
Suppose \eqref{eq:Finvlipschitz} is satisfied with a constant $L_1>0.$ Then
\begin{align*}
p_n(f_1-f_2 ) & =\sum_{j=0}^n 2^{ja}2^{-ja} 2^{jd(\frac{1}{2}-\frac{1}{p})} \|Q_j(f_1-f_2)\|_p \\ 
  & \leq \left(\sum_{j=0}^n 2^{ja} \right) \|f_1-f_2\|_{-a,p,\infty} \\ 
 & \leq L_1 (2^a-1)^{-1}  2^{a(n+1)} \|F(f_1)-F(f_2)\|_{\Yspace}.
\end{align*}
Hence we put $L_\nu:= L_1 2^a (2^a-1)^{-1}$ to obtain the second statement. 
For the converse implication we assume \eqref{eq:strategy_operator} holds with $\nu_n = L_\nu 2^{na}$. Then 
\begin{align*}
2^{-na} 2^{nd(\frac{1}{2}-\frac{1}{p})} \|Q_n(f_1-f_2)\|_p \leq 2^{-na} p_n(f_1-f_2)\leq L_\nu \|F(f_1)-F(f_2)\|_{\Yspace}
\end{align*} 
Taking supremum on the left hand side yields \Cref{Finvlipschitz} with $L_1:=L_\nu$.
\end{proof} 
\begin{lemma}\label{besov:varitheorem}
Suppose \eqref{eq:p_in_12}, $a>0$, \Cref{Finvlipschitz} and 
$f^\dagger\in B_{p,\infty}^s\cap D$ with $\|f^\dagger \|_{s,p,\infty}\leq \varrho$ 
for some $s>0$. Then $f^{\dagger}$ satisfies the variational source condition 
\begin{align}\label{eq:varicondbesov}
 \|f-f^\dagger \|_{0,p,1} \leq \|f\|_{0,p,1}-\|f^\dagger \|_{0,p,1} +\tilde{\varphi} (\|F(f)-F(f^\dagger)\|_{\Yspace}^2)
\quad \text{ for all }  f\in D.
\end{align}
with concave index function $\tilde{\varphi}$ which is bounded by  
\[ \tilde{\varphi}(t)\leq c \varrho^{\frac{a}{s+a}} t^{\frac{s}{2(s+a)}}  \quad\text{for all } 0\leq t\leq \left(\frac{s\varrho}{a}\right)^2.
\]
Here the constant $c>0$ depends only on $a,s$ and $L_1$.
\end{lemma}
\begin{proof} 
Let $f\in B^0_{p,1}$ and $n\in\nat$. As the functionals
$p_n^\perp(f):=\|f\|_{0,p,1}-p_n(f)$ are given by 
 \[p_n^\perp(f)=\sum_{j=n+1}^\infty 2^{jd(\frac{1}{2}-\frac{1}{p})} \|Q_jf \|_p,\] 
they are seminorms, i.e.\ assumption \eqref{eq:strategy_additive} is satisfied.  
We estimate $p_n^\perp(f^\dagger)$ by the Jackson type inequality 
\begin{align*}
 p_n^\perp(f^\dagger)&  = \sum_{j=n+1}^\infty 2^{-js}2^{js} 2^{jd(\frac{1}{2}-\frac{1}{p})} \|Q_jf \|_p \\ 
 & \leq \left( \sum_{j=n+1}^\infty 2^{-js}\right) \|f\|_{s,p,\infty}\leq (1-2^{-s})^{-1} 2^{-(n+1)s} \varrho.
\end{align*}
%\begin{align*}
%2^{jd(\frac{1}{2}-\frac{1}{p})} \|Q_j (f-f^\dagger)\|_p & \leq  2^{jd(\frac{1}{2}-\frac{1}{p})} |\Lambda_j|^{\frac{1}{p}-\frac{1}{2}}  \|Q_j  (f-f^\dagger)\|_2 \\
%& \leq C_\Lambda^{\frac{1}{p}-\frac{1}{2}} 2^{ja} \| f-f^\dagger\|_{-a,2,2} \\ 
%& \leq LC_\Lambda^{\frac{1}{p}-\frac{1}{2}} 2^{ja} \|F(f)-F(f^\dagger)\|_{\Yspace}.   
%\end{align*}
%Hence 
%\begin{align*}
%p_n(f-f^\dagger)  & \leq  LC_\Lambda^{\frac{1}{p}-\frac{1}{2}} \left( \sum_{j=0}^n 2^{ja}\right)\|F(f)-F(f^\dagger)\|_{\Yspace} \\ 
%& \leq  LC_\Lambda^{\frac{1}{p}-\frac{1}{2}}(2^a-1)^{-1}  2^{a(n+1)} \|F(f)-F(f^\dagger)\|_{\Yspace}.
%\end{align*} 
%This proves \eqref{eq:strategy_operator}.
%\Cref{vari:theorem} implies \eqref{eq:varicondbesov} with 
In particular this verifies \eqref{eq:strategy_limit}.
\Cref{Finvlipschitz} and \Cref{lemma:equivalentassumptions} imply \eqref{eq:strategy_operator} with $\nu_n=L_1 (2^a-1)^{-1}  2^{a(n+1)}$.
We define $\tilde{C}:=2 \max\left\{ (1-2^{-s})^{-1} , L_1(2^a-1)^{-1} \right\}$. \Cref{vari:theorem} yields \eqref{eq:varicondbesov} with
\begin{align}\label{eq:index_inf} \tilde{\varphi}(t)=  \tilde{C} \inf_{n\in \nat} \left( 2^{a(n+1)}  \sqrt{t} + 2^{-s(n+1)} \varrho \right) .
\end{align}
It remains to estimate $\tilde{\varphi}.$
%\Cref{lemma:coercive} and \Cref{norm_of_R_n} yield a constant $\tilde{L}>0$ such that 
%\[ \|P_n ( f_1 -f_2) \|_{0,p,1} \leq \tilde{L} 2^{a(n+1)} \|F(f_1)-F(f_2)\|_{\Yspace} \quad \text{ for all } f_1, f_2 \in D.\]   
%The function $f^\dagger$ has a unique decomposition $f^\dagger=\sum_{j\in\mathbb{N}_0} \sum_{l=1}^{L_j} \sum_{m\in M_j}\lambda_{j,m}^l \phi_{j,m}^l$. We estimate $\|(I-P_n)f^\dagger\|_{0,p,1}$ by a jackson type inequality: 
%\begin{align*}
%\|(I-P_n)f^\dagger\|_{0,p,1} & =\sum_{j=n+1}^\infty \sum_{l=1}^{L_j}  2^{-js}2^{js}2^{jd \left(\frac{1}{2}-\frac{1}{p}\right)} \left(\sum_{m\in M_j} |\lambda_{j,m}^l|^p\right)^\frac{1}{p} \\
%&\leq (2^d-1) \left(\sum_{j=n+1}^\infty  2^{-js} \right) \|f^\dagger\|_{s,p,\infty}\\
%&\leq \frac{2^d-1}{1-2^{-s}} 2^{-s(n+1)} \varrho.
%\end{align*}
%We define $\tilde{C}=2\max \left\{ \frac{2^d-1}{1-2^{-s}}, \tilde{L}\right\}.$
%\Cref{vari:theorem} implies \eqref{eq:varicondbesov} with 
%\begin{align}\label{eq:index_inf} \tilde{\varphi}(t)=  \tilde{C} \inf_{n\in \mathbb{N}} \left( 2^{a(n+1)}  t + 2^{-s(n+1)} \varrho \right) .\end{align}
%It remains to estimate $\tilde{\varphi}.$
Since the stated inequality is trivially fulfilled if $t=0$ we may assume $0<t\leq \left( \frac{s\varrho}{a}\right)^2.$ Let $h\colon (0,\infty)\rightarrow \mathbb{R}$ given by $h(x):=x^a \sqrt{t}+\varrho x^{-s}.$ Then $x_0:= \left(\frac{s\varrho}{a \sqrt{t}}\right)^{1/(a+s)}$ is the global minimum of $h$, and $h$ is increasing on $[x_0,\infty).$ The assumption on $t$ implies $x_0\geq 1$. Let $N\in\nat$ be minimal with $x_0\leq 2^{N+1}$. Then $2^{N+1}\leq 2x_0$. We estimate 
\begin{align*}
\tilde{\varphi}(t)&\leq \tilde{C} h(2^{N+1})\leq \tilde{C} h(2x_0)\leq 2^{a}\tilde{C} h(x_0) \\ &= 2^{a}\tilde{C}\left(\left(\frac{s}{a}\right)^\frac{a}{a+s}+\left( \frac{s}{a}\right)^{-\frac{s}{a+s}}\right) \varrho^\frac{a}{a+s}t^\frac{s}{2(a+s)}.\qedhere
\end{align*}
\end{proof}
If we use $\loss(f,f^{\dagger}) = \frac{1}{2}\|f-f^\dagger \|_{0,p,1}$ rather 
than $\loss(f,f^{\dagger}) = \|f-f^\dagger \|_{0,p,1}$ as loss function and use \Cref{Flipschitz}
we obtain a global estimate of the index function with the same exponent. 
\begin{theorem} \label{vari_thm_global}
Suppose \eqref{eq:p_in_12}, $a\geq d/p-d/2$ and Assumptions \ref{Finvlipschitz} and \ref{Flipschitz} hold true. Assume $s>0$ and 
$f^\dagger\in B_{p,\infty}^s\cap D$ with $\|f^\dagger \|_{s,p,\infty}\leq \varrho$.
Then the variational source condition 
\begin{align} \frac{1}{2}\|f-f^\dagger \|_{0,p,1} 
\leq \|f\|_{0,p,1}-\|f^\dagger \|_{0,p,1} 
+\varphi\left(\|F(f)-F(f^\dagger)\|_{\Yspace}^2\right)
\quad \text{for all }  f\in D\label{varicond:modi}
\end{align} 
is satisfied with $\varphi(t)=c_\varphi \varrho^{\frac{a}{s+a}} t^{\frac{s}{2(s+a)}}$  
and a constant $c_\varphi>0$ depending only on $a,s,L_1,L_2$. 
\end{theorem} 
\begin{proof} 
Let $\tilde{\varphi}, \tilde{C}$ and $c$ be as in \Cref{besov:varitheorem} and its proof. We define \[K:=4 L_2(1-2^{-s})^{-1} \quad\text{and } c_\varphi= \max \left\{ c, \tilde{C}(2^a K+2^{-s})\left(\frac{s}{a}\right)^{-\frac{s}{s+a}} \right\}
\]
and make a case distinction for $t:=\|F(f)-F(f^\dagger)\|_{\Yspace}^2$. \\
\emph{Case 1:} $0\leq t\leq (K \varrho)^2$. We will show that 
$\tilde{\varphi}(t)\leq \varphi(t)$ such that 
eq.~\eqref{varicond:modi} 
holds true by \Cref{besov:varitheorem}. \\
\emph{Case 1a:} $0\leq t\leq \left(\frac{s\varrho}{a}\right)^2$. 
Here $\tilde{\varphi}(t)\leq \varphi(t)$ follows immediately from the definitions. \\
\emph{Case 1b:} $\left(\frac{s\varrho}{a}\right)^2<t\leq (K \varrho)^2$.   
By taking $n=0$ in \eqref{eq:index_inf} we obtain 
\[ \tilde{\varphi}(t)\leq \tilde{\varphi}((K \varrho)^2)\leq \tilde{C}(2^aK+2^{-s})\varrho \leq c_\varphi\varrho^{\frac{a}{s+a}} \left(\frac{s\varrho}{a}\right)^{\frac{s}{s+a}} \leq \varphi(t).\]
\medskip
\emph{Case 2:} $t>(K\varrho)^2$. 
%We prove \eqref{varicond:modi} by a case distinction on $f\in D$. \\
%\emph{Case 2a:} 
If $\|f\|_{0,p,1}\leq 3 \|f^\dagger \|_{0,p,1}$ 
then \Cref{Flipschitz}, \ref{prop:embeddings}.\ref{embed_mixed}. and   \ref{prop:embeddings}.\ref{embed_s}. yields 
\begin{align*}
\|F(f)-F(f^\dagger) \|_{\Yspace} & \leq L_2 \|f-f^\dagger \|_{-a,2,1}\leq L_2 \|f-f^\dagger \|_{0,p,1} \leq 4 L_2 \|f^\dagger \|_{0,p,1}   \leq K \varrho 
\end{align*}
such that we are in Case 1. 
%Hence \Cref{besov:varitheorem} implies 
%\begin{align*}
% \frac{1}{2}\|f-f^\dagger \|_{0,p,1} & \leq\|f-f^\dagger \|_{0,p,1}\leq \|f\|_{0,p,1}-\|f^\dagger \|_{0,p,1} +\tilde{\varphi}\left(\|F(f)-F(f^\dagger)\|_{\Yspace}^2\right) \\
%& \leq \|f\|_{0,p,1}-\|f^\dagger \|_{0,p,1} +\varphi (\|F(f)-F(f^\dagger)\|_{\Yspace}^2). 
%\end{align*}
%\emph{Case 2b:} 
Hence $\|f \|_{0,p,1} > 3 \|f^\dagger \|_{0,p,1}$, and we obtain 
\begin{align*}
 \frac{1}{2}\|f-f^\dagger \|_{0,p,1} & \leq  \frac{1}{2}\|f\|_{0,p,1}+ \frac{1}{2}\|f^\dagger \|_{0,p,1} 
  = \|f\|_{0,p,1} + \frac{1}{2}\left( \|f^\dagger \|_{0,p,1}-\|f \|_{0,p,1}\right) \\
 & \leq \|f\|_{0,p,1} + \frac{1}{2}\left( \|f^\dagger \|_{0,p,1}-3 \|f^\dagger \|_{0,p,1}\right) 
  = \|f\|_{0,p,1}- \|f^\dagger\|_{0,p,1} \\
 & \leq  \|f\|_{0,p,1}-\|f^\dagger \|_{0,p,1} +\varphi (\|F(f)-F(f^\dagger)\|_{L^2}^2).\qedhere
\end{align*}
\end{proof}

\section{Convergence analysis for deterministic errors} \label{convergence}
%Before deriving error bounds for the estimators $\hat{f}_{\alpha}$ in 
%\eqref{Tyk} under \Cref{assumption:operator}, we note that 
%a global minimizer of the Tikhonov functional exists by standard arguments 
%(see e.g.\ \cite[Thm.~3.22]{Scherzer_etal:09}), and it is unique due to 
%strict convexity if $F$ is linear. For the existence proof we use the norm topology 
%of $B^{-a}_{2,2}$ as weaker topology on $\Xspace=B^0_{p,1}$ and the compactness 
%of the embedding $B^0_{p,1}\hookrightarrow B^{-a}_{2,2}$ as discussed before 
%\Cref{assumption:operator}.
\subsection{existence of minimizers} \label{existence of minimizers}
%Before we prove convergence rates we discuss the existence and uniqueness of minimizers in \eqref{Tyk}.
\begin{proposition}\label{existence_of_minimizer}
Assume \eqref{eq:completeness}, \eqref{eq:p_in_12}, \eqref{eq:domain_closed} 
and \Cref{Flipschitz}. 
Then for every $\gobs\in \Yspace$ and every $\alpha>0$ there exists a global minimizer of the Tikhonov functional in \eqref{Tyk}. 
%If $F$ is linear, then the global minimizer is unique. 
\end{proposition}
\begin{proof}
Let $(f_m)_{m\in \mathbb{N}} \subset D$ be a minimizing sequence of the Tikhonov functional. Then $(f_m)$ is bounded in $B_{p,1}^0$. Due to $a> d/p-d/2$ the embedding $B_{p,1}^0 \subset B^{-a}_{2,1}$ is compact (see \Cref{compactness_embedding}). Hence there exists a subsequence $(f_{m_k})$ such that $f=\lim_k f_{m_k}$ exists in $B^{-a}_{2,1}$. This implies $\lim_k Q_j(f_{m_k})= Q_j(f)$ for all $j\in \nat$. We obtain  
\[ p_n(f) = \lim_k p_n(f_{m_k})\leq \liminf_k \|f_{m_k}\|_{0,p,1}\quad \text{for all}\quad  n\in \nat.\] Hence $\|f\|_{0,p,1}\leq \liminf_k \|f_{m_k}\|_{0,p,1} <\infty$. Therefore, $f\in B^0_{p,1}$ and \eqref{eq:domain_closed} yields $f\in D$. \Cref{Flipschitz} implies 
 $\lim_k F(f_{m_k}) =F(f)$ in $\Yspace$. 
 Therefore, $f$ is a global minimum of the Tikhonov functional. 
%\\
%$B_{p,1}^0 \subset \Xspace_D$ implies $\lim_k Q_jf_{m_k} =Q_jf$ for all $j\in\nat$. Therefore we obtain 
%\[ p_n (f)  = \lim_k p_n (f_{m_k}) \leq \liminf_k \|f_{m_k}\|_{0,p,1} \quad\text{for all } n\in \mathbb{N}.\]
%Taking the limit $n\to \infty$
%The uniqueness in the linear case follows from the strict convexity of $\|\cdot \|_{L^2}^2$.
\end{proof}
\subsection{upper bounds on the worst case error}
To obtain deterministic convergence rates let us first recall the general 
error bound under a variational source condition \eqref{eq:varicond}. 
Let $\mathcal{S}:\Yspace  \to\mathbb{R}$ be the noisy data fidelity functional such that the data term appearing in the Tikhonov functional becomes $\frac{1}{2}\mathcal{S}(F(f))$. In this paper we 
only consider 
\begin{subequations}
\begin{align}
\label{eq:S_det}
&\mathcal{S}(g):=\|g-\gobs\|^2
&&\mbox{for the noise model \eqref{deterministic_noise}}\\ 
\label{eq:S_rand}
&\mathcal{S}(g):=\|g\|^2-  2 \langle g,\gobs\rangle
&&\mbox{for the noise model \eqref{statistical_noise}.}
\end{align} 
\end{subequations}
The effective noise level $\mathbf{err}:\Yspace\to \mathbb{R}$ is defined by 
\begin{align}\label{eq:defi_err}
\err(g):= \mathcal{S}(g^{\dagger})-\mathcal{S}(g) + \frac{1}{C_{\mathbf{err}}}
\|g-g^{\dagger}\|^2.
\end{align}
For non-quadratic data fidelity term, $\|g-g^{\dagger}\|^2$ should be replaced by some non-quadratic noise-free data fidelity terms  
in \eqref{eq:defi_err} and \eqref{eq:varicond} to keep the effective noise level small. 
For the deterministic noise model \eqref{deterministic_noise} the convenient choice 
of the constant is $C_{\err}=2$, which leads to 
\begin{align}\label{eq:uniform_err_bound}
\err(g) \leq \delta^2 \qquad \mbox{for all }g\in \Yspace.
\end{align}
We need the following variant of \cite[Thm.~3.1 and Cor.~3.1]{grasmair:10}:
\begin{proposition}\label{prop:standard_conv}
Let $F:D\subset \Xspace\to \Yspace$ be some forward operator between Banach spaces $\Xspace$ and $\Yspace$, 
let $\mathcal{S}$ be given by either \eqref{eq:S_det} or \eqref{eq:S_rand}, and  
let $\err$ be defined by \eqref{eq:defi_err}. Moreover, 
assume that $f^{\dagger}\in \Xspace$ satisfies the variational source condition 
\eqref{eq:varicond} for a concave index function $\varphi$ and that 
\begin{align}\label{eq:det_err_bound}
\hat{f}_{\alpha}\in\argmin_{f\in D}\left[\frac{1}{2}\mathcal{S}(F(f))+\alpha 
\|f\|_{\Xspace}\right]
\end{align}
for $\alpha>0$. Then:
\begin{enumerate}
\item For all $\alpha>0$ the error bounds 
\begin{subequations}
\begin{align}
\label{eq:error_decomp_X}
& \loss\left(\hat{f}_{\alpha},f^{\dagger}\right) 
\leq \frac{1}{2\alpha} \err\left(F(\hat{f}_{\alpha})\right) + \psi(2\alpha C_{\err}), \\
\label{eq:error_decomp_Y}
& \left\|F\left(\hat{f}_{\alpha}-g^{\dagger}\right)\right\|^2 
\leq 2C_{\err}\err\left(F(\hat{f}_{\alpha})\right) + 4C_{\err}\alpha 
\psi(4 C_{\err}\alpha)
\end{align}
\end{subequations}
hold true with the Fenchel conjugate $\psi(\alpha):=(-\varphi)^*\left(- \frac{1}{\alpha}\right)$. 
\item If  \eqref{eq:uniform_err_bound} holds true  and if we choose 
$-\frac{1}{2C_{\err}\alpha^*}\in \partial(-\varphi)(C_{\err}\delta^2)$, then 
\begin{align}\label{eq:gen_det_rate}
\loss\left(\hat{f}_{\alpha^*},f^{\dagger}\right) 
\leq C_{\err} \varphi\left(\delta^2\right).
\end{align}
\end{enumerate}
\end{proposition}
\begin{proof}
\eqref{eq:error_decomp_X} and \eqref{eq:gen_det_rate} follow from 
\cite[Thm.~3.3]{Werner2012} and its proof if the scaled Bregman divergence is 
replaced by the more general loss function everywhere. Analogously, 
\eqref{eq:error_decomp_Y} follows from \cite[Thm.~2.3]{Hohage2014}.
\end{proof}

With these preparations the error bound for deterministic errors follows easily:
\begin{theorem}\label{rate:deterministic}
Suppose \eqref{eq:p_in_12}, $a\geq d/p-d/2$ and Assumptions  \ref{Finvlipschitz} and \ref{Flipschitz} hold ture. Assume $s>0$ and $f^\dagger\in B_{p,\infty}^s\cap D$ with $\|f^\dagger \|_{s,p,\infty}\leq \varrho$.
\begin{enumerate}
\item For exact data $\gobs= g^\dagger$ every global minimizer $f_\alpha$ of \eqref{Tyk} satisfies 
\begin{align}\label{rate:deterministic:exact}
\| f_\alpha - f^\dagger\|_{0,p,1}\leq C_r \varrho^{\frac{2a}{s+2a}} \alpha^{\frac{s}{s+2a}}.
\end{align}  
\item  Let $\delta>0$ and $\gobs\in \Yspace$ be given by \eqref{deterministic_noise}. If $\alpha$ is chosen by $\alpha(\delta)=\delta^{\frac{s+2a}{s+a}}$, then every global minimizer $\hat{f}_\alpha$ of \eqref{Tyk} satisfies 
\begin{align}\label{rate:deterministic:noisy}
\| \hat{f}_\alpha - f^\dagger\|_{0,p,1}\leq C_r \varrho^{\frac{a}{s+a}} \delta^{\frac{s}{s+a}}.
\end{align}  
\end{enumerate}
Here $C_r>0$ denotes is a constant depending only on $a,s,L_1,L_2$. 
\end{theorem}
\begin{proof}
\Cref{vari_thm_global} and \Cref{prop:standard_conv} with $C_\err=2$ imply the error bound 
\begin{align}\label{rate:err}
\frac{1}{2}\| \hat{f}_\alpha - f^\dagger\|_{0,p,1}\leq \frac{1}{2\alpha}\mathbf{err}\left(F( \hat{f}_\alpha)\right)+ \psi( 4\alpha) 
\end{align} 
with  $\psi\colon (0,\infty)\rightarrow (0,\infty)$  given by
\begin{align}\label{eq:psi_bound}
\psi(t)=\left(-\varphi\right)^\ast\left(-\frac{1}{t}\right)
=C_\psi \varrho^{\frac{2a}{s+2a}} t^{\frac{s}{s+2a}}
\end{align} 
for some constant $C_\psi>0$ depending on $s$, $a$ and $c_\varphi$ from \Cref{vari_thm_global}. 
In the case of exact data we have $\mathbf{err}(g)=0$. Hence \eqref{rate:err}  implies \eqref{rate:deterministic:noisy}.
The error bound \eqref{rate:deterministic:noisy} follows from the second statement in \Cref{prop:standard_conv}.
\end{proof}

\begin{remark}\label{rem:Lp_bound}
If $B^0_{p,1}\left(\{\dframe_\lambda\}\right)= B^0_{p,1}(\Omega)$  then it follows from  the continuity of the embedding $B^0_{p,1}(\Omega)\hookrightarrow L^p(\Omega)$ (see \cite{Triebel2010}) 
that under the assumptions of \Cref{rate:deterministic} we also have an error bound 
with respect to the $L^p$-norm, 
\begin{align}\label{eq:Lp_bound}
\|\hat{f}_\alpha - f^\dagger\|_{L^p}=  \mathcal{O} 
\left(\varrho^{\frac{a}{s+a}} \delta^{\frac{s}{s+a}}\right).
\end{align}
For $p=2$, this allows in particular to compare the 
sparsity promoting regularization \eqref{Tyk} to standard Tikhonov-regularization 
with the quadratic penalty term $\alpha\|f\|_{L^2}^2$: 
Whereas for quadratic Tikhonov regularization the error bound 
\eqref{eq:Lp_bound} is only valid for $s\leq 2a$ (see \cite{EHN:96}), 
the sparsity promoting regularization \eqref{Tyk} obeys 
this error bound for arbitrarily high smoothness parameters $s>0$ given sufficient 
smoothness of the wavelet systems, i.e.\ $\overline{s}=\infty$. 
\end{remark}

\subsection{lower bounds on the worst case error}\label{sec:lowerbounds}
To prove that the rate \eqref{rate:deterministic:noisy} is order optimal, we need to show lower bounds on the convergence rate. 
Recall that the \emph{modulus of continuity} $\omega(\delta,K)$ of the inverse of $F$
restricted to some set $K\subset B^0_{p,1}$  with respect to a norm 
$\|\cdot\|_{\mathrm{loss}}$ is given by 
\[ \omega\left(\delta,K,\|\cdot\|_{\mathrm{loss}}\right)
=\sup\left\{\|f_1-f_2\|_{\mathrm{loss}} \colon f_1,f_2 \in K\quad\text{with } \|F (f_1)-F(f_2)\|_{\Yspace}\leq \delta \right\}.
\]
The worst case error $\Delta(\delta,K,\|\cdot\|_{\mathrm{loss}})$ of a (linear or nonlinear) reconstruction method \linebreak $R\colon \Yspace\rightarrow B^0_{p,1}$ on $K$ with respect to $\|\cdot\|_{\mathrm{loss}}$ is given by 
\[ \Delta_R(\delta,K,\|\cdot\|_{\mathrm{loss}}):=\sup\left\{ \|f-R(\gobs)\|_{\mathrm{loss}} \colon f\in K,\, \gobs\in \Yspace \text{ with } \|F(f)-\gobs\|_{\Yspace}\leq \delta\right\}.\] 
One can prove that the modulus of continuity serves as a lower bound on the worst case error 
of any reconstruction method (see {\cite[Rem. 3.12]{EHN:96}}, {\cite[Lemma 2.8]{weidling2018optimal}}):
\begin{equation}\label{worstcaseerror} 
\Delta_R\left(\delta,K,\|\cdot\|_{\mathrm{loss}}\right)\geq 
\frac{1}{2} \omega\left(2\delta,K,\|\cdot\|_{\mathrm{loss}}\right).
\end{equation}
We will choose $K$ as a Besov ball 
\begin{align}\label{compact_source}
 K^s_{p,q}\left(\bar{f},\varrho\right):=\left\{ f\in B^s_{p,q} \colon \|f-\bar{f}\|_{s,p,q}\leq \varrho\right\}
\end{align}
for some $s> 0$, $q\in [1,\infty]$, $\bar{f}\in B^s_{p,q}$ and $\varrho>0$. 
%We  obtain $ K^s_{p,q}\left(\bar{f},\varrho\right)\subset B^0_{p,1}$ from the embeddings $B^s_{p,q}\subset B^s_{p,\infty}\subset B^0_{p,1}$.
\begin{theorem}\label{theo:lower_bound}
Suppose \eqref{eq:p_in_12}, that $(\dframe_\lambda)_{\lambda\in \Lambda}$ is a Riesz basis of $L^2(\Omega)$, $F$ satisfies \Cref{Flipschitz} 
and that the domain $D$ of $F$ contains a Besov ball 
$K^s_{p,q}(\bar{f},\varrho)$ 
for some $q\in [1,\infty]$, $s> 0$, $\bar{f}\in B^s_{p,q}$ and $\varrho>0$.
There cannot exist a reconstruction method $R\colon \Yspace \rightarrow B^0_{p,1}$ for the operator $F$ satisfying the worst case error bound
\begin{align} 
\Delta_R\left(\delta,K^s_{p,q}\left(\bar{f},\varrho\right),\|\cdot\|_{0,p,1}\right)
=o\left(\varrho^{\frac{a}{s+a}} \delta^{\frac{s}{s+a}}\right).
\end{align}
Hence in this case the rate in \Cref{rate:deterministic} is optimal up to the value of the constant. 
\end{theorem}

\begin{proof}
Since $(\dframe_{\lambda})_{\lambda\in \Lambda}$ is a Riesz basis, the map 
$L^2(\Omega) \rightarrow \ell^2(\Lambda)$ given by $f\mapsto (\langle f,\dframe_\lambda\rangle)_{\lambda\in \Lambda}$ is surjective.\\
Let $j\in\nat$. We set $\beta:=C_{\Lambda}\inv 2^{-\frac{1}{2}jd}\min\{   2^{-js}\varrho,  L_2^{-1} 2^{ja}\delta \}$. There exists $f_j\in L^2(\Omega)$ such that $(Q_n(f_j))_k= \beta\delta_{n,j}$ for all $k\in \Lambda_n$, $n\in \nat.$ We obtain
\begin{align}\label{eq:optimal_prior_norm}
 &\|f_j\|_{s,p,q}= \beta 2^{js} 2^{jd(\frac{1}{2}-\frac{1}{p})} |\Lambda_j|^{\frac{1}{p}} \leq \beta C_{\Lambda}^{\frac{1}{p}}    2^{js} 2^{\frac{1}{2}jd}\leq \varrho,\\
& \|F(\bar{f}+f_j)-F(\bar{f})\|_{\Yspace} \leq L_2 \|f_j\|_{-a,2,1}\leq L_2 \beta C_{\Lambda}^{\frac{1}{2}} 2^{-ja} 2^{\frac{1}{2}jd}\leq \delta
\nonumber.
\end{align}
%Consider \[f_j:= \beta \sum_{m\in M_j} \phi_{j,m}^1\in B^s_{p,\infty}.\] Then 
%\begin{align}\label{eq:optimal_prior_norm}
 %&\|f_j\|_{s,p,\infty}=\beta 2^{js} 2^{jd(\frac{1}{2}-\frac{1}{p})} |M_j|^{\frac{1}{p}}= \beta  2^{js} 2^{\frac{1}{2}jd}\leq \varrho ,\\
%& \|F(\bar{f}+f_j)-F(\bar{f})\|_{\Yspace} \stackrel{\eqref{Flipschitz}}{\leq} L \|f_j\|_{-a,2,2}= L\beta 2^{-ja} |M_j|^{\frac{1}{2}}\leq  L\beta 2^{-ja} 2^{\frac{1}{2}jd}\leq \delta
%\nonumber
%\end{align}
%and 
We have 
\begin{align}\label{eq:optimal_loss_norm}
 \|f_j\|_{0,p,1}= \beta 2^{jd(\frac{1}{2}-\frac{1}{p})} |\Lambda_j|^{\frac{1}{p}} \geq  C_{\Lambda}\inv \min\{  2^{-js}\varrho, L_2^{-1} 2^{ja} \delta\}.
\end{align}
This shows that 
\[ \omega\left(\delta,K^s_{p,\infty}(\bar{f},\varrho),\|\cdot\|_{0,p,1}\right) 
\geq \sup_{j\in\mathbb{N}_0} \|f_j\|_{0,p,1}
\geq C_{\Lambda}\inv \sup_{j\in\mathbb{N}_0} \left( \min\{ \varrho 2^{-js},\delta L_2^{-1} 2^{ja}\} \right). \] 
%The two terms are balanced for $2^j\approx \left(\frac{\varrho L}{\delta}\right)^\frac{1}{s+a}.$
Suppose that $\delta<\varrho L_2$ and let $j_0\in\mathbb{N}$ be minimal with $2^{j_0}> \left(\varrho L_2/\delta\right)^{1/(s+a)}$. Then 
$2^{j_0}\leq 2 \left(\varrho L_2/\delta\right)^{1/(s+a)}$ and we obtain 
\[ \omega\left(\delta,K^s_{p,\infty}(\bar{f},\varrho),\|\cdot\|_{0,p,1}\right)
\geq C_{\Lambda}\inv  2^{-j_0 s}\varrho\geq C_{\Lambda}\inv  2^{-s} \varrho \left(\frac{\varrho L_2}{\delta}\right)^\frac{-s}{s+a}= c  \varrho^{\frac{a}{s+a}} \delta^{\frac{s}{s+a}}\]
with $c:= C_{\Lambda}\inv 2^{-s} L_2^{-s/(s+a)}$. Finally eq.~\eqref{worstcaseerror} yields 
\[ \Delta_R\left(\delta,K^s_{p,\infty}(\bar{f},\varrho),\|\cdot\|_{0,p,1}\right)
\geq \frac{1}{2} \omega\left(2\delta,K^s_{p,\infty}(\bar{f},\varrho),\|\cdot\|_{0,p,1}\right) 
\geq  \frac{c}{2} 2^{\frac{s}{s+a}} \varrho^{\frac{a}{s+a}} \delta^{\frac{s}{s+a}}.\qedhere\]
\end{proof}

\begin{remark}\label{rem:lower_Lp_bound}
%\begin{enumerate}
%\item
If $\{\dframe_{\lambda}\}$ is a frame of compactly supported wavelets, then 
there exists constants $C_1,C_2>0$ such that for all $j\in \nat$ 
the inequalities 
%\todo{only the first of those bounds needed here? Ja, ich finde es aber trotzdem 
%klarer, beide Ungleichungen aufzuschreiben.}
\[C_1 \|f_j\|_{L^p}\leq \beta 2^{dj(\frac{1}{2}-\frac{1}{p})}|\Lambda_j|^{1/p}%\|f_j\|_{0,p,1}
\leq C_2 \|f_j\|_{L^p}\]
hold true for the functions $f_j$ in the proof of \Cref{theo:lower_bound} (see \cite[Thm.~3.3.1]{cohen:03}). Then one can modify 
eq.~\eqref{eq:optimal_loss_norm} 
in the previous proof  to show that there cannot be a reconstruction method 
satisfying 
\[
\Delta_R\left(\delta,K^s_{p,q}\left(\bar{f},\varrho\right),\|\cdot\|_{L^p}\right)
=o\left(\varrho^{\frac{a}{s+a}} \delta^{\frac{s}{s+a}}\right).
\]
Hence also the upper bound in \Cref{rem:Lp_bound} is of optimal order. 
%\item 
%\Cref{theo:lower_bound} also remains valid for Besov balls 
%$K^{\tilde{s}}_{\tilde{p},\tilde{q}}(\bar{f},\rho)$ with arbitrary 
%$\tilde{q}\in [1,\infty]$, $s> 0$ and $p\in [1,\infty]$ such that 
%$\tilde{s}-\frac{d}{\tilde{p}}<s-\frac{d}{p}$ by obvious modification of 
%eq.~\eqref{eq:optimal_prior_norm}. Whereas for $\tilde{q}= \infty$ this can 
%directly be derived from the continuity of the embedding 
%$B^s_{p,q}\hookrightarrow B^{\tilde{s}}_{\tilde{p},q}$, the case $\tilde{q}<\infty$ 
%requires an inspection of the proof as $B^s_{p,q}\subset B^s_{p,\infty}$. 
%\end{enumerate}
\end{remark}

To finish this section we study smoothness described by the total variation 
seminorm on the unit interval, 
\[
|f|_{\mathrm{TV}} := \sup\left\{ \sum_{j=1}^k|f(t_{j-1})-f(t_j)| \colon 
0=t_0<t_1<\ldots<t_k=1, k\in\mathbb{N}\right\}.
\]
Recall that $\|f\|_{\mathrm{BV}}:= \|f\|_{L^1}+ |f|_{\mathrm{TV}}$. 
\begin{corollary}\label{coro:tv}
Let $(\dframe_\lambda)_{\lambda\in \Lambda}$ be a wavelet Riesz basis of $L^2([0,1])$ and suppose that $\Omega = [0,1]$, $B^0_{p,1}(\{\dframe_\lambda\})=B^0_{p,1}([0,1])$. Assume that 
\eqref{eq:p_in_12}, $a\geq d/p-d/2$ and Assumptions  
\ref{Finvlipschitz} and \ref{Flipschitz} hold true 
and that the domain $D$ of $F$ contains a ball 
\[K_{\mathrm{BV}}(0,\varrho):=
\{f\in \mathrm{BV}([0,1])\colon\|f\|_\mathrm{BV}<\varrho\},\qquad \varrho>0.
\]
Then there exists a constant $C>0$ independent of $\varrho$ and $\delta$ such 
that the worst case error of the reconstruction method $R$ defined 
by \eqref{Tyk} with $p=1$ satisfies 
\begin{align} 
 C^{-1}\varrho^{\frac{a}{1+a}} \delta^{\frac{1}{1+a}}
\leq \Delta_R\left(\delta,K_{\mathrm{BV}}\left(0,\varrho\right),\|\cdot\|_{L^1}\right)
\leq  C\varrho^{\frac{a}{1+a}} \delta^{\frac{1}{1+a}}
\end{align}
for sufficiently small $\delta>0$. 
\end{corollary}

\begin{proof}
We have the continuous embeddings $B^1_{1,1}([0,1])\hookrightarrow 
\mathrm{BV}([0,1])\hookrightarrow B^1_{1,\infty}([0,1])$ 
(see e.g.~\cite[Prop.~4.3.21 and 4.3.22]{GN:15}).  
The former embedding and \Cref{theo:lower_bound}/\Cref{rem:lower_Lp_bound} imply the lower bound, 
and the latter embedding and \Cref{rate:deterministic}/\Cref{rem:Lp_bound} 
imply the upper bound. 
\end{proof}

\section{Sparsity bounds and a converse result}\label{sec:sparsity}
In this section we first derive an estimate of the (block-)sparsity of the minimizers 
$\hat{f}_\alpha$ for deterministic noise and use this to prove a converse result 
to the error bound for exact data. 
\begin{theorem}\label{theo:sparsity} 
Suppose that \eqref{eq:Q_surjective} and \eqref{eq:p_in_12} hold true.
Let $\tilde{p}\in [p,2]$, $a\in [d/p-d/ \tilde{p}, \overline{s})$ and   
$D^\circ$ be the interior of $D$ in $B^0_{p,1}.$ Assume that $F\colon D^\circ\rightarrow \Yspace$ 
is Fr\'echet differentiable. 
Moreover, assume that $\Yspace$ is uniformly smooth (e.g.\ a Hilbert space) 
or otherwise $F$ is linear. 
Let $\gobs\in \Yspace$, $\alpha>0$ and $\hat{f}_\alpha\in D^\circ$ be a local minimizer of 
\eqref{Tyk} and suppose that the operator norm
$K:=\norm{F'[\hat{f}_{\alpha}]}{B^{-a}_{\tilde{p},1}\to \Yspace}$ is finite.  
Then for all $j\in\mathbb{N}_0$ we have 
\[ Q_j \hat{f}_\alpha = 0\qquad \text{if }
\qquad \frac{K}{\alpha}   
 \|F(\hat{f}_\alpha) -\gobs\|_\Yspace < 2^{j(a + \frac{d}{\tilde{p}}-\frac{d}{p})}.
\]
In particular, at most a finite number of wavelet coefficients of $\hat{f}_\alpha$ 
do not vanish. 
\end{theorem} 
\begin{proof}
We first study the first order optimality conditions for $\hat{f}_{\alpha}$. 
If $F$ is nonlinear and $\Yspace$ is uniformly smooth, we use \Cref{lem:optimality} 
with $\mathcal{F}(f)= \frac{1}{2}\|F(f)-\gobs\|_{\Yspace}^2$,  
$\mathcal{G}(f)= \alpha\|f\|_{0,p,1}$, and $\Xspace = B^0_{p,1}$. 
Recall that the uniform smoothness of 
$\Yspace$ implies that the mapping $\mathcal{T}(g):=\frac{1}{2}\|g\|_{\Yspace}^2$ 
is Fr\'echet differentiable, and its derivative is the duality 
mapping $j:\Yspace\to \Yspace'$, $j(g):= \mathcal{T}'[g]$ (see \cite[I.\S 3, Thm.~3.12 and I.\S 4. Cor.~4.5]{cioranescu:90}). 
%As the duality mapping 
%$J:= \partial \left(\frac{1}{2}\| \cdot\|_{\Yspace}^2\right):\Yspace\to\Yspace^\prime$ 
%is single-valued and continuous \todo{make precise}, 
Therefore, $\mathcal{F}$ is Gateaux differentiable with 
\[ \mathcal{F}'[\hat{f}_\alpha]h
=\langle j(F(\hat{f}_\alpha) -\gobs),F'[\hat{f}_\alpha]h\rangle
= \langle F'[\hat{f}_\alpha]^\ast j(F(\hat{f}_\alpha) -\gobs),h\rangle,\] 
and the first order optimality condition becomes 
\[
0\in F'[\hat{f}_{\alpha}]^\ast j\left(F(\hat{f}_\alpha) -\gobs \right)+  \alpha\partial \|\cdot \|_{0,p,1}(\hat{f}_\alpha).
%\label{extremalrel}
\]
If $\Yspace$ is not smooth, the duality mapping is a set-valued map 
$J:=\partial \mathcal{T}$ with $J(g) = \{j(g)\}$ in the smooth case. 
However, since we assume $F$ to be linear in this case, $\mathcal{F}$ is convex and 
continuous, and by virtue of the sum and the chain rule for subgradients of 
convex functionals (see \cite{ET:76}), there exist in both cases  
\begin{align} \label{eq:extremal_simple}
g_\alpha^\ast\in J\left(F(\hat{f}_\alpha) -\gobs \right) \quad \text{and } 
f_\alpha^\ast\in \partial\|\cdot \|_{0,p,1}(\hat{f}_\alpha)  \quad
\text{ such that }  F'[\hat{f}_\alpha]^\ast g_\alpha^\ast= - \alpha f_\alpha^\ast.
\end{align}
We have $\|g_\alpha^\ast\|_{\Yspace^\prime}
=\|F(\hat{f}_\alpha) -\gobs\|_\Yspace$.
As $(B^0_{p,1})^\prime=B^0_{p^\prime,\infty}$ (see \Cref{dual_spaces_iso}), 
Asplund's theorem (see 
\cite[I.\S 4, Thm.~4.4]{cioranescu:90}) yields
\begin{align*}
\partial\|\cdot \|_{0,p,1}(\hat{f}_\alpha)= \left\{f^*_{\alpha}\in B^0_{p^\prime,\infty} \colon \sum\nolimits_{j=0}^\infty \left\langle Q_j f^*_{\alpha}, Q_j\hat{f}_\alpha\right\rangle
=\|\hat{f}_\alpha\|_{0,p,1}\text{ and } \|f^*_{\alpha}\|_{0,p^\prime,\infty}=1\right\}.
\end{align*}
By H\"older's inequality we have 
\[   
\left\langle Q_jf_\alpha^\ast, Q_j\hat{f}_\alpha\right\rangle 
\leq \|Q_jf_\alpha^\ast\|_{p^\prime} \cdot\|Q_j\hat{f}_\alpha\|_p 
\leq 2^{jd(\frac{1}{2}-\frac{1}{p})}  \|Q_j \hat{f}_\alpha\|_p % \leq \|\hat{f}_\alpha\|_{0,p,1}
\]
for every $j\in \nat$.  
Hence $\left\langle Q_jf_\alpha^\ast, Q_j\hat{f}_\alpha\right\rangle =2^{jd(\frac{1}{2}-\frac{1}{p})}  \|Q_j \hat{f}_\alpha\|_p$ for all $j\in\nat$ 
and %we obtain the following implication:
\begin{align}\label{eq:implication}
 Q_j \hat{f}_\alpha = 0 \qquad 
\text{if }  \|Q_j f_\alpha^\ast\|_{p^\prime} < 2^{jd(\frac{1}{2}-\frac{1}{p})}.
\end{align}
Since we assume $F'[\hat{f}_\alpha]$ to have a continuous extension to 
an operator from $B^{-a}_{\tilde{p},1}$ to $\Yspace$, its adjoint 
$F'[\hat{f}_{\alpha}]^\ast$ maps $\Yspace^\prime$ boundedly into 
$(B^{-a}_{\tilde{p},1})^\prime=B^a_{\tilde{p}^\prime,\infty}$ with $1/\tilde{p}+1/\tilde{p}^\prime=1$, and 
$\norm{F'[\hat{f}_{\alpha}]^\ast}{\Yspace^\prime\to B^{a}_{\tilde{p}^\prime,\infty}}
=\norm{F'[\hat{f}_{\alpha}]}{B^{-a}_{\tilde{p},1}\to \Yspace} = K$. 
Together with \eqref{eq:extremal_simple} we obtain  
\begin{align*}
\snorm{Q_jf_\alpha^\ast}_{p^\prime} 
& = \frac{1}{\alpha} \snorm{Q_jF'[\hat{f}_{\alpha}]^\ast g^\ast_\alpha}_{p^\prime} 
\leq \frac{1}{\alpha} \snorm{Q_jF'[\hat{f}_{\alpha}]^\ast g^\ast_\alpha}_{\tilde{p}^\prime} \\
&\leq \frac{1}{\alpha} 2^{-ja}2^{jd\left(\frac{1}{2}-\frac{1}{\tilde{p}}\right)} \snorm{F'[\hat{f}_{\alpha}]^\ast g^\ast_\alpha}_{a,\tilde{p}^\prime,\infty} 
\leq  \frac{K}{\alpha} 2^{-ja}2^{jd\left(\frac{1}{2}-\frac{1}{\tilde{p}}\right)} 
\snorm{g^\ast_\alpha}_{\Yspace^\prime} \\ 
&=  \frac{K}{\alpha} 2^{-ja} 2^{jd\left(\frac{1}{2}-\frac{1}{\tilde{p}}\right)}  
\snorm{F(\hat{f}_\alpha) -\gobs}_\Yspace.
\end{align*}
Inserting this into \eqref{eq:implication} yields the result. 
\end{proof}

Based on \Cref{theo:sparsity} we prove a converse result in the case of exact data $\gobs=g^\dagger$.
For $\alpha>0$ let  % unique minimizers of the corresponding Tikhonov functional will be denoted by 
\begin{align}\label{Tyk_exactdata}
f_\alpha \in \argmin_{f\in D} \left[ \frac{1}{2} \|F(f) -g^\dagger\|_{\Yspace}^2 
+ \alpha \|f\|_{0,p,1} \right].
\end{align}
\begin{corollary}\label{theo:converse}
%Suppose that $(\dframe_\lambda)_{\lambda\in \Lambda}$ is a Riesz basis for %$L^2(\Omega)$, \eqref{eq:setting} and $a>0$. Let $F=\linOp \colon B^{-a}_{p,1}\rightarrow \Yspace$ be a linear and bounded satisfying
%\[ 
%\|f\|_{-a,p,\infty}\leq L_1 \|\linOp f \|_{\Yspace} \quad\text{for all }f\in .+
%B^{0}_{p,1} \] 
Suppose that \eqref{eq:Q_surjective}, \eqref{eq:p_in_12}, \Cref{Finvlipschitz}, $a>0$ and that $D$ is closed in the topology induced by $\|\cdot\|_{-a,p,1}$ on $B^0_{p,1}$.
Assume that $F\colon D^\circ\rightarrow \Yspace$ 
is Fr\'echet differentiable and that $\Yspace$ is uniformly smooth 
or otherwise $F$ is linear. Suppose that 
$K:=\sup_{f\in D^\circ}\norm{F'[f]}{B^{-a}_{p,1}\to \Yspace}$ is finite. Let $f^\dagger\in B^0_{p,1}$, $g^\dagger=F(f^\dagger)$ and assume 
all minimizers $f_\alpha$, $\alpha>0$ in \eqref{Tyk_exactdata}  
belong to $D^\circ$ 
(e.g. $D = B^0_{p,1}$). Then the following statements are equivalent: 
\begin{enumerate}
\item\label{it:reg} $f^\dagger\in B^s_{p,\infty}$
\item\label{it:rate}  $\gamma:=\sup_{\alpha>0} 
 \alpha^{-\frac{s}{s+2a}} \|f_\alpha  - f^\dagger\|_{0,p,1} <\infty$
\end{enumerate} 
More precisely, we can bound $\gamma\leq C_r \|f^{\dagger}\|_{s,p,\infty}^{2a/(s+2a)}$
and $\|f^{\dagger}\|_{s,p,\infty}\leq (\sqrt{2}K)^{s/a}\gamma^{(2a+s)/2a}$ with the constant $C_r$ from \Cref{rate:deterministic}.
\end{corollary} 
\begin{proof}
The mean value theorem and the boundedness assumption on the derivative of $F$ implies 
\begin{align}
\|F(f_1)-F(f_2)\|_{\Yspace}\leq K \|f_1-f_2\|_{-a,p,1}\quad\text{for all } f_1, f_2\in D. \label{eq:Flipschitzp}
\end{align}
Since $a>0$ the same arguments as in \Cref{existence_of_minimizer}  imply the existence of minimizers in \eqref{Tyk_exactdata}.
The implication $\ref{it:reg} \Rightarrow \ref{it:rate}$ follows from 
\Cref{rate:deterministic}. (note that \Cref{Flipschitz} can be replaced by \eqref{eq:Flipschitzp} in the proof of \Cref{vari_thm_global}). \\
To prove the implication $\ref{it:rate} \Rightarrow \ref{it:reg}$,  
 we will first estimate $\|F(f_\alpha) -g^\dagger\|_\Yspace$ and then use 
\Cref{theo:sparsity} with $\tilde{p}=p$. Let $\alpha>0$. Since $f_\alpha$ is a solution to \eqref{Tyk_exactdata} we obtain 
\begin{align*}
 \frac{1}{2} \|F(f_\alpha)-g^\dagger \|_{\Yspace}^2 
\leq \alpha \left( \|f^\dagger\|_{0,p,1}-\|f_a\|_{0,p,1} \right)
\leq \alpha \|f_\alpha- f^\dagger \|_{0,p,1} \leq \gamma \alpha^{\frac{2(s+a)}{s+2a}}\,.  
\end{align*}
\Cref{theo:sparsity} yields  
\[ 
 Q_j f_\alpha=0\qquad \text{if }\quad  \sqrt{2\gamma}K\alpha^{\frac{-a}{s+2a}} < 2^{ja}. 
\] 
Hence, given $j\in \mathbb{N}_0$ we have
\begin{align*}
2^{jd\left(\frac{1}{2}-\frac{1}{p}\right)}\| Q_j f^\dagger \|_{p} & =  2^{jd\left(\frac{1}{2}-\frac{1}{p}\right)}\| Q_j (f_\alpha  - f^\dagger)\|_p  
 \leq \| f_\alpha  - f^\dagger\|_{0,p,1}\leq \gamma  \alpha^\frac{s}{s+2a}.  
\end{align*}
for all $\alpha>0$ satisfying 
$K\sqrt{2\gamma}\alpha^{\frac{-a}{s+2a}} < 2^{ja}$,
%$K 2^{-ja} \sqrt{2 \gamma} < \alpha^\frac{a}{s+2a}$, 
using the assumed convergence rate once more.  
Letting $\alpha$ tend to  $\left( K 2^{-ja} \sqrt{2 \gamma} \right)^\frac{s+2a}{a}$ from above we obtain 
\[ 2^{jd\left(\frac{1}{2}-\frac{1}{p}\right)} \| Q_j f^\dagger \|_p\leq \gamma \left(\left( K 2^{-ja} \sqrt{2 \gamma} \right)^\frac{s+2a}{a}\right)^\frac{s}{s+2a}= \gamma^{\frac{2a+s}{2a}}  (\sqrt{2}K)^\frac{s}{a} 2^{-sj}. 
\]
This implies 
\[ \| f^\dagger \|_{s,p,\infty}= \sup_{j\in \nat}2^{sj} 2^{jd\left(\frac{1}{2}-\frac{1}{p}\right)} \|Q_j f^\dagger \|_{p} \leq \gamma^{\frac{2a+s}{2a}} (\sqrt{2}K)^\frac{s}{a} <\infty. 
\]
Hence $f^\dagger\in B^s_{p,\infty}.$
\end{proof}

\begin{remark} 
The assumptions of \Cref{theo:converse} are only natural if 
$\Yspace=L^p(\Omegaobs)$ with the same $p$ as in \eqref{eq:p_in_12} and 
\eqref{Tyk_exactdata}. However, if data error bounds are given with respect to 
the $L^2(\Omegaobs)$-norm, then $\Yspace$ should be chosen as $L^2(\Omegaobs)$ 
which also has computational advantages. Then $\tilde{p}=2$ is natural in 
\Cref{theo:sparsity}, and the corresponding assumptions are equivalent to \Cref{Flipschitz}. 
But for $\tilde{p}=2$ the obtained sparsity bound only yields a converse result for $p=2$ 
with the previous proof. It remains an open question if for $p<\tilde{p}$ the 
approximation error bound \eqref{rate:deterministic:exact} is valid on a larger set 
than $B^s_{p,\infty}$. 

Of course converse result for the error bound \eqref{rate:deterministic:noisy} with 
noisy data would be of interest as well. 
This seems to require lower bounds on the effect of data noise, which are beyond the 
scope of this paper. 
\end{remark}

\section{Convergence analysis for random noise}
In this section we let $a>\frac{d}{2}$. It suffices to require $(\dframe_\lambda)_{\lambda\in \Lambda}$ 
to be a frame. 
We will make the following assumption on the noise process.   
\begin{assumption}\label{noise_assumption}
Let $p^\prime\in [2,\infty]$ be conjugate to $p$ in \eqref{Tykstat}, 
i.e.~$\frac{1}{p}+\frac{1}{p^\prime}=1$. We assume that there 
exists $\tilde{d}\in [0,2a)$ such that $Z$ is a random variable with 
\[
Z\in B^{-\tilde{d}/2}_{p^\prime,\infty}(\Omegaobs)
\quad\text{almost surely. }\] 
\end{assumption} 
%Gaussian white noise $W \colon \Omegaobs \rightarrow \mathcal{D}^\prime(\Omegaobs)$ satisfies \Cref{noise_assumption} (see \cite[Thm. 3.4]{Veraar2011}(1)). We exclude $\tilde{p}=\infty$ as $W$ does not belong to $B^{-d/2}_{\infty,\infty}(\Omegaobs)$ with probability $1$ (see \cite[Thm. 3.4 (6)]{Veraar2011}). We give up smoothness to handle the case $\tilde{p}=\infty$: 
For Gaussian white noise on $\Omegaobs=(\mathbb{R}/\mathbb{Z})^d$ this assumption is satisfied with $\tilde{d}=d$ if 
$p>1$, but only with 
$\tilde{d}>d$ if $p=1$ (see \cite[Thm.~3.4]{Veraar2011}). 
For Gaussian white noise and the Brownian bridge prozess on $\Omegaobs=[0,1]$ the assumption is satisfied with $\tilde{d}=d$ if $p>1$ (see \cite[Thm.~4.4.3]{GN:15}) and it follows with $\tilde{d}>d$  for $p=1$ from the embedding $B^{-d/2}_{\tilde{p},\infty}(\Omegaobs) \subset B^{-\tilde{d}/2}_{\infty,\infty}(\Omegaobs)$ for $\tilde{p}\in [1,\infty)$ such that $-\frac{\tilde{d}}{2}>-\frac{d}{2} -\frac{d}{\tilde{p}}$ (see \cite[\S 3.3.1]{Triebel2010})). 
%(The fact that $\tilde{d}=d$ is not possible for 
%$p=1$ follows from \cite[Thm.~3.4(6)]{Veraar2011}, and that $\tilde{d}>d$ is 
%possible is a consequence of the embedding 
%
%\todo[inline]{Dass die Voraussetzung mit  $\tilde{d}>d$ für $p=1$ erfüllt ist,  könnte man auch direkt aus Veraar2011 zitieren \cite[Thm.~3.4(3)]{Veraar2011}, dann wird die Klammer etwas kürzer. \emph{OK, wir sollten allerdings noch erw\"ahnen, dass 
%Veraar nur den periodischen Fall behandelt.}}
We have introduced the variable $\tilde{d}$ to treat the two cases simultaneously. 
Our setting even covers much more general noise processes. 

%\begin{proposition} \label{lower_the_d}
%Suppose \Cref{noise_assumption} and let $\tilde{d}\in (0,\infty)$ with $\tilde{d}>d$. Then $Z\in B^{-\tilde{d}/2}_{\infty,\infty}(\Omegaobs)$ almost surely. 
%\end{proposition}
%\begin{proof}
%There exists $p\in [1,\infty)$ such that $-\frac{\tilde{d}}{2}>-\frac{d}{2} -\frac{d}{p}$. By  \cite[\S 3.3.1]{Triebel2010} this yields an embedding $B^{-d/2}_{p,\infty} \subset B^{-\tilde{d}/2}_{\infty,\infty}$.
%\todo{include this embedding in appendix}
%\end{proof}

The following lemma will be used to guarantee existence of minimizers of 
\eqref{Tykstat} and to prove statistical convergence rates. 
It uses standard Besov norms of functions on $\Omegaobs$. In order to 
distinguish them from the Besov-type norms $\|\cdot\|_{s,p,q}$ of functions 
on $\Omega$ defined via wavelet coefficients and to avoid nested indices, 
we will use the notation $\norm{g}{B^s_{p,q}(\Omegaobs)}$ 
or $\norm{g}{B^s_{p,q}}$ instead of $\|g\|_{B^s_{p,q}(\Omegaobs)}$.

\begin{lemma}\label{interpolation}
Suppose $p\in[1,2]$ and $0<s<a$, then 
\[ \norm{g}{B^s_{p,1}} \leq C 
\norm{g}{L^2}^{1-\frac{s}{a}} \norm{g}{B^a_{p,1}}^\frac{s}{a}\quad\text{for all } g\in B^a_{p,1}(\Omegaobs)\cap L^2(\Omegaobs)\] 
with a constant $C>0$ independent of $g$. 
\end{lemma}
\begin{proof}
By interpolation theory (see \cite[3.3.6]{Triebel2010}) %there exists a constant $C_0>0$ such that 
%\[ \norm{g}{s,p,1} \leq C_0 \norm{g}{0,p,2}^{1-\frac{s}{a}} \norm{g}{a,p,1}^\frac{s}{a}\quad\text{for all } g\in B^a_{p,1}\cap L^2.\] \todo{argue with besov sequence space}
this holds true if $\norm{g}{L^2}$ is replaced by 
$\norm{g}{B^0_{p,2}}$. 
The claim follows as $p\leq 2$ implies $\norm{g}{B^0_{p,2}}\leq 
\norm{g}{B^0_{2,2}} = \norm{g}{L^2}$. 
\end{proof}

\subsection{existence of minimizers}
\begin{proposition} 
%Let $a>\frac{d}{2}$.
Assume \eqref{eq:completeness}, \eqref{eq:p_in_12}, \eqref{eq:domain_closed} and Assumptions \ref{Flipschitz}, \ref{operator_additional} and \ref{noise_assumption}.
Let $\alpha>0$, and suppose $\gobs$ is given by \eqref{statistical_noise}. Then the Tikhonov functional in \eqref{Tykstat} has a solution $\hat{f}_\alpha$ almost surely.
\end{proposition}
\begin{proof}[Proof. (sketch)]
Let $\tilde{d}\in (0,\infty)$ such that $\frac{d}{2}< \frac{\tilde{d}}{2}<a $. 
We denote the H\"older conjugate of $p$ by $p^\prime\in [1,\infty]$. 
By \Cref{noise_assumption} %and \Cref{lower_the_d} imply $Z\in B^{-\tilde{d}/2}_{p^\prime,\infty}$ almost surely. Hence we find 
we have
\(
N:= \norm{\gobs}{B^{-\tilde{d}/2}_{p^\prime,\infty}} < \infty
\) almost surely. 
As in the proof of \cite[Prop.~4.8]{weidling2018optimal} we can bound the data fidelity 
term from below using eqs.~\eqref{interpolation} and \eqref{operator_additional} 
to obtain the following 
lower bound on the Tikhonov functional with positive constants $A$ and $c$:
%Let $f\in D$. We use $(B^{\tilde{d}/2}_{p,1})^\prime=  B^{-\tilde{d}/2}_{p^\prime,\infty}$, \Cref{interpolation}, \eqref{operator_additional} and Young's inequality $xy\leq \frac{1}{2} x^\frac{4a}{2a-\tilde{d}}+ c y^\frac{4a}{2a+\tilde{d}}$ with a generic constant $c$ to estimate 
%\begin{align*}
%\langle \gobs, F(f)\rangle & \leq N \|F(f)\|_{\tilde{d}/2,p,1}  \leq CN\|F(f)\|_{L^2}^{1-\tilde{d}/2a} \|F(f)\|_{a,p,1}^{\tilde{d}/2a} \\
%& \leq \frac{1}{2} \|F(f)\|_{L^2} + c N^{\frac{4a}{2a+\tilde{d}}} \left( \|F(f)-F(0)\|_{a,p,1} +\|F(0)\|_{a,p,1} \right)^\frac{2\tilde{d}}{2a+\tilde{d}} \\
%& \leq \frac{1}{2} \|F(f)\|_{L^2} + c N^{\frac{4a}{2a+\tilde{d}}} \|f\|_{0,p,1}^\frac{2\tilde{d}}{2a+\tilde{d}} + A
%\end{align*}
%with $A= c N^{\frac{4a}{2a+\tilde{d}}} \|F(0)\|_{a,p,1}^\frac{2\tilde{d}}{2a+\tilde{d}}$ and a generic constant $c$. We proceed by estimating the Tikhonov functional 
\begin{align*}
 \frac{1}{2} \norm{F(f)}{L^2}^2 -\langle \gobs, F(f)\rangle +\alpha  \|f\|_{0,p,1} \geq -c N^{\frac{4a}{2a+\tilde{d}}} \|f\|_{0,p,1}^\frac{2\tilde{d}}{2a+\tilde{d}} + \alpha  \|f\|_{0,p,1}- A.
\end{align*}
Since $\frac{2\tilde{d}}{2a+\tilde{d}}< 1$ the right hand side tends to $\infty$ as $\|f\|_{0,p,1}\rightarrow \infty.$ Therefore,  a minimizing sequence of the Tikhonov functional in \eqref{Tykstat} is bounded in $B^0_{p,1}$ almost surely. 
With the same arguments as in the proof \Cref{existence_of_minimizer} we conclude the existence of a global minimum. 
\end{proof} 

\subsection{convergence rates}
Using the variational source condition and the interpolation inequality in 
\Cref{interpolation} we obtain the following error bound: 
\begin{theorem}\label{theorem_stat_rates}
Suppose \eqref{eq:p_in_12},  and Assumptions \ref{Finvlipschitz}, \ref{Flipschitz}, \ref{operator_additional} and \ref{noise_assumption} hold true.
Let $\gobs$ given by \eqref{statistical_noise}, let $s>0$, and suppose that 
$f^\dagger\in B_{p,\infty}^s$ with $\|f^\dagger \|_{s,p,\infty}\leq \varrho$. 
If $\alpha$ is chosen such that 
$\alpha\sim \varrho^\frac{\tilde{d}-2a}{\tilde{d}+2a+2s}$, 
%Then we obtain to following error bounds: 
%\begin{enumerate}
%\item If $p\in (1,2]$ with H\"older conjugate $p^\prime$ and $\alpha$ is chosen by $\alpha\sim \varrho^\frac{d-2a}{d+2a+2s}\varepsilon^\frac{2s+4a}{d+2a+2s}$ then there exists a constant $C>0$ independent of $f^\dagger$, $\hat{f}_\alpha$, $\varepsilon$ and $\varrho$ such that every global minimizer $\hat{f}_\alpha$ of \eqref{Tykstat} satisfies
%\begin{align}    \| \hat{f}_\alpha-f^\dagger \|_{0,p,1}  \leq C  \varrho^\frac{2a+d}{2a+2s+d} \epsilon^\frac{2s}{2a+2s+d} \left( 1+  \| Z\|_{-d/2,p^\prime,\infty}^\frac{4a}{2a-d}\right).\label{stat_rate_p}
%\end{align}
%\item If $p=1$, $\tilde{d}\in (0,\infty)$ such that $\frac{d}{2}< \frac{\tilde{d}}{2}<a$ and $\alpha$ is chosen by $\alpha\sim \varrho^\frac{\tilde{d}-2a}{\tilde{d}+2a+2s}\varepsilon^\frac{2s+4a}{\tilde{d}+2a+2s}$
then there exists a constant $C>0$ independent of $f^\dagger$, $Z$, 
$\varepsilon$ and $\varrho$ such that every global minimizer $\hat{f}_\alpha$ of \eqref{Tykstat}  satisfies
\begin{align}    \| \hat{f}_\alpha-f^\dagger \|_{0,p,1}  \leq C  \varrho^\frac{2a+\tilde{d}}{2a+2s+\tilde{d}} \epsilon^\frac{2s}{2a+2s+\tilde{d}} \left( 1+  \norm{Z}{B^{-\tilde{d}/2}_{p^\prime,\infty}}^\frac{4a}{2a-\tilde{d}}\right).\label{stat_rate}
\end{align}
%\end{enumerate}
\end{theorem}

\begin{proof}
As in the proof of \Cref{rate:deterministic}, 
\Cref{vari_thm_global} and \Cref{prop:standard_conv}  imply the error bound 
\eqref{rate:deterministic:noisy} on $\|\hat{f}_\alpha - f^\dagger\|_{0,p,1}$ as well as the bound 
%\begin{align}\label{rates:err}
%\frac{1}{2}\| \hat{f}_\alpha - f^\dagger\|_{0,p,1}\leq \frac{1}{2 \alpha}\mathbf{err}\left(F( \hat{f}_\alpha)\right)+ \psi( 2\alpha)  \quad\text{and}
%\end{align}
\begin{align}\label{rates:err_image}
 \norm{F( \hat{f}_\alpha)-g^\dagger}{L^2}^2\leq 2 \mathbf{err}\left(F( \hat{f}_\alpha)\right)+4 \alpha\psi( 4\alpha) 
\end{align}
%Here  $\psi\colon (0,\infty)\rightarrow (0,\infty)$ is given by 
with $\psi(t) =C_\psi \varrho^{\frac{2a}{s+2a}} t^{\frac{s}{s+2a}}$ (see \eqref{eq:psi_bound}). 
%\[ \psi(t)=\left(-(\varphi\circ \sqrt{\cdot})\right)^\ast\left(-\frac{1}{t}\right)=c \varrho^{\frac{2a}{s+2a}} t^{\frac{s}{s+2a}}\] for some constant $c>0$ depending on $s$, $a$ and $\tilde{C}$ from \Cref{vari_thm_global}. 
As the effective noise level is given by
\[ \mathbf{err}(g)=\mathcal{S}(g^\dagger)-\mathcal{S}(g) +  \norm{g-g^\dagger}{L^2}^2=2 \varepsilon\langle Z , g-g^\dagger \rangle
\] 
with  $\mathcal{S}(g)=  \|g\|^2_{L^2} -2  \langle \gobs , g\rangle$, 
it can be bounded by 
%Now let us first consider the case $p\in (1,2].$
%We establish a bound on the effective noise level $\mathbf{err}$ and use \eqref{rates:err}. It is natural to estimate 
\[ \mathbf{err}(F(\hat{f}_\alpha))
\leq 2\varepsilon \norm{Z}{B^{-\tilde{d}/2}_{p^\prime,\infty}}\cdot 
\norm{F(\hat{f}_\alpha)-g^\dagger}{B^{\tilde{d}/2}_{p,1}}. 
\]
To estimate the second factor we use first \Cref{interpolation} and \eqref{operator_additional} to obtain 
\begin{align*}
\norm{F(\hat{f}_\alpha)-g^\dagger}{B^{\tilde{d}/2}_{p,1}} 
&  \leq C \norm{F(\hat{f}_\alpha)-g^\dagger}{L^2}^{1-\tilde{d}/2a} 
\norm{F(\hat{f}_\alpha)-F(f^\dagger)}{B^a_{p,1}}^{\tilde{d}/2a} \\ 
&\leq C \norm{F(\hat{f}_\alpha)-g^\dagger}{L^2}^{1-\tilde{d}/2a} 
\|\hat{f}_\alpha-f^\dagger\|_{0,p,1}^{\tilde{d}/2a}
\end{align*}
with a generic constant $C>0.$ The image space error bound \eqref{rates:err_image} yields 
\begin{align*}
 \mathbf{err}(F(\hat{f}_\alpha))&\leq 2C \norm{\varepsilon Z}{B^{-\tilde{d}/{2}}_{p^\prime,\infty}}\cdot \norm{F(\hat{f}_\alpha)-g^\dagger}{L^2}^{1-\tilde{d}/2a} \cdot \|\hat{f}_\alpha-f^\dagger\|_{0,p,1}^{\tilde{d}/2a} \\
 &\leq 2C \norm{\varepsilon Z}{B^{-\tilde{d}/{2}}_{p^\prime,\infty}}\cdot  \|\hat{f}_\alpha-f^\dagger\|_{0,p,1}^{\tilde{d}/2a}  \cdot \left[ 4 \mathbf{err}\left(F( \hat{f}_\alpha)\right)+4 \alpha\psi( 4\alpha)\right]^{\frac{2a-\tilde{d}}{4a}} \\
 &\leq C  \norm{\varepsilon Z}{B^{-\tilde{d}/{2}}_{p^\prime,\infty}}^\frac{4a}{2a+\tilde{d}}\cdot  \|\hat{f}_\alpha-f^\dagger\|_{0,p,1}^{\frac{2\tilde{d}}{2a+\tilde{d}}} +\frac{1}{2}  \left[  \mathbf{err}\left(F( \hat{f}_\alpha)\right)+ \alpha\psi( 4\alpha)\right],
\end{align*}
where the last step follows from Young's inequality and requires a replacement of the constant $C$. 
Subtracting the $\mathbf{err}$ term yields an upper bound 
\[ \mathbf{err}(F(\hat{f}_\alpha))
\leq C  \norm{\varepsilon Z}{B^{-\tilde{d}/{2}}_{p^\prime,\infty}}^\frac{4a}{2a+\tilde{d}}\cdot  \|\hat{f}_\alpha-f^\dagger\|_{0,p,1}^{\frac{2\tilde{d}}{2a+\tilde{d}}} +\alpha\psi( 4\alpha).\]
Inserting this bound into \eqref{rate:deterministic:noisy} we deduce 
\begin{align*}
 \frac{1}{2} \| \hat{f}_\alpha-f^\dagger \|_{0,p,1} 
& \leq \frac{C}{\alpha}  \norm{\varepsilon Z}{B^{-\tilde{d}/{2}}_{p^\prime,\infty}}^\frac{4a}{2a+\tilde{d}}\cdot  \|\hat{f}_\alpha-f^\dagger\|_{0,p,1}^{\frac{2\tilde{d}}{2a+\tilde{d}}} + \psi( 4\alpha)+ \psi(2\alpha) \\ 
 &\leq \frac{C}{\alpha}  \norm{\varepsilon Z}{B^{-\tilde{d}/2}_{p^\prime,\infty}}^\frac{4a}{2a+\tilde{d}}\cdot  \|\hat{f}_\alpha-f^\dagger\|_{0,p,1}^{\frac{2\tilde{d}}{2a+\tilde{d}}} + 2 \psi( 4\alpha) \\ 
 &  \leq \frac{1}{4}  \|\hat{f}_\alpha-f^\dagger\|_{0,p,1} + C \alpha^{\frac{\tilde{d}+2a}{\tilde{d}-2a}}  \norm{\varepsilon Z}{B^{-\tilde{d}/2}_{p^\prime,\infty}}^\frac{4a}{2a-\tilde{d}}+  2 \psi( 4\alpha).
\end{align*}
Here again the last step follows from Young's inequality and requires a replacement of the constant $C$. 
Rearranging terms and using the explicit expression for $\psi$ we obtain a constant $C>0$ such that 
\[ \| \hat{f}_\alpha-f^\dagger \|_{0,p,1}  \leq C \left( \alpha^{\frac{\tilde{d}+2a}{\tilde{d}-2a}}  \norm{\varepsilon Z}{B^{-\tilde{d}/2}_{p^\prime,\infty}}^\frac{4a}{2a-\tilde{d}} +\alpha^{\frac{s}{s+2a}}  \varrho^{\frac{2a}{s+2a}}.  \right) \qedhere\]
Inserting the parameter choice rule implies \eqref{stat_rate}.
%The same computations as above with $p=1$, $p^\prime=\infty$ and $d$ replaced by $\tilde{d}$ prove \eqref{stat_rate_1}.
\end{proof}

\Cref{theorem_stat_rates} provides not only a rate in terms of the noise level 
parameter $\varepsilon$, but also a bound of the distributation of the error 
in terms of the distribution of a negative Besov norm of the noise. 
In particular, we obtain bounds on the expectation of arbitrary moments of 
the error if the Besov norm of the noise satisfies a 
a large deviation inequality of the form 
\begin{align}\label{deviation_ineq}
\mathbb{P} \left[ \norm{Z}{B^{-\tilde{d}/2}_{p^\prime,\infty}} > M_Z+ \tau \right] \leq \exp(-C_Z \tau^\mu) \quad\text{for all } \tau>0.
\end{align}
with constants $C_Z, M_Z, \mu>0$. 
By \cite[Cor.~3.7]{Veraar2011}) such an inequality is satisfied in particular 
for Gaussian white noise with $\tilde{d}=d$ for $p>1$ and with 
$\tilde{d}>d$ for $p=1$.

\begin{corollary}\label{cor:stat_rates} 
If \eqref{deviation_ineq} holds true in addition to the assumptions of 
\Cref{theorem_stat_rates}, then 
\begin{align} \label{rate_expectation_1}  \mathbb{E}\left( \| \hat{f}_\alpha -f^\dagger\|_{0,p,1}^\sigma\right)^{1/\sigma} \leq \tilde{C} \varrho^{\frac{2a+\tilde{d}}{2a+2s+\tilde{d}}} \varepsilon^{\frac{2s}{2a+2s+\tilde{d}}}
\end{align} 
holds true for any $\sigma\geq 1$. In particular, if $Z$ is Gaussian white noise, 
then this inequality holds true with $\tilde{d}=d$ for $p\in (1,2]$ and 
with $\tilde{d}>d$ for $p=1$. 
\end{corollary}
\begin{proof}
%First we consider the case $p\in (1,2].$
It suffices to show that $\mathbb{E}\left[(1+\|Z\|^t)^\sigma\right]<\infty$ for all $t>0$ 
where $\|Z\|$ denotes one of the Besov norm appearing in \eqref{stat_rate}.
Due to the inequality $(a+b)^\sigma\leq 2^\sigma(a^\sigma+b^\sigma)$ for $a,b\geq 0$ 
this reduces to showing that $\mathbb{E}\left[\|Z\|^t\right]<\infty$ for all $t>0$ which 
can be deduced from \eqref{deviation_ineq} as follows:
\begin{align*}
\mathbb{E}\left[\|Z\|^t\right]
&\leq M_Z^t\mathbb{P}\left[\|Z\|\leq M_Z\right] 
+ \sum_{j=1}^\infty (M_Z+j)^t\mathbb{P}\left[j-1<\|Z\|-M_Z\leq j\right] \\
&\leq M_Z^t + \sum_{j=1}^\infty (M_Z+j)^t\exp(-C_Zj^\mu) <\infty
\qedhere
\end{align*}
%\Cref{theorem_stat_rates} implies the existence of a constant $c>0$ such that
%\[ \mathbb{P}\left( \| \hat{f}_\alpha -f^\dagger\|_{0,p,1} > (c+t) \varrho^{\frac{2a+d}{2a+2s+d}} \varepsilon^{\frac{2s}{2a+2s+d}}\right) \leq \exp\left(-C_W t^{\mu(\frac{1}{2}-\frac{d}{4a})} \right)\]
%for all $t>0$. Integration yields \eqref{rate_expectation_p}.
%If $p=1$, we use an embedding as in the proof of \Cref{lower_the_d} to obtain $\|W\|_{-\tilde{d}/2,\infty,\infty} \leq K \|W\|_{\frac{d}{2},\tilde{p},\infty}$ for some $\tilde{p} \in [1,\infty)$. This implies \eqref{derivation_ineq} with $p^\prime=\infty$ and $d$ replaced by $\tilde{d}$ and suitable constants $C_W$, $M_W$, $\mu>0$. The same argument as for $p\neq 1$ finishes the proof.
\end{proof}

\subsection{lower bounds} 
The upper bound expected reconstruction error in \Cref{cor:stat_rates}  
coincides with the following lower bound for $p>1$, and it almost coincides 
for $p=1$:
\begin{proposition}
Suppose that $Z$ in \eqref{statistical_noise} is white noise, 
\Cref{operator_additional} and \Cref{operator_s} hold true for some $s>0$, and that 
$F(B^s_{p,\infty}\cap D)$ contains an open ball in $B^{s+a}_{p,\infty}(\Omegaobs)$ 
around $g_0=F(f_0)$. Then there exists a constant $c>0$ such that 
\[
\inf_{R} \sup_{\|f^{\dagger}-f_0\|_{s,p,\infty}\leq \rho} 
\mathbb{E}\left\|R(F(f^{\dagger})+\varepsilon Z)-f^{\dagger}\right\|_{0,p,1}
\geq c \rho^{\frac{2a+d}{2s+2a+d}} \varepsilon^{\frac{2s}{2s+2a+d}}
\]
for sufficiently small $\rho$ 
where the infimum is taken over all measurable mappings $R$ from 
$B^{-d/2}_{p,\infty}(\Omegaobs)$ to $B^0_{p,1}$.  
\end{proposition}
%\todo[inline]{Braucht man hier noch Annahmen aus Gleichung (7)? Nein!}
\begin{proof}
%Suppose that 
%$\{g\in B^{s+a}_{p,\infty}(\Omegaobs):\norm{g-g_0}{B^{s+a}_{p,\infty}(\Omegaobs)}<\rho_0\}
%\subset F(B^s_{p,\infty}\cap D)$ and $g_0=F(f_0)$.
We start from lower bounds for the estimation of $g^{\dagger}=F(f^{\dagger})$, 
see \cite[Thms.~7,9]{DJKP:95}, \cite[Cor.~4.12]{weidling2018optimal}. 
\[
\inf_{S} \sup_{\norm{g^{\dagger}-g_0}{B^{s^{**}}_{p,\infty}(\Omegaobs)}\leq \tilde{\rho}} 
\mathbb{E}\left[\norm{g^{\dagger}-S(g^{\dagger}+\varepsilon Z)}{B^{s^*}_{p,1}(\Omegaobs)}\right]
\geq \tilde{c} \tilde{\rho}^{\frac{2s^*+d}{2s^{**}+d}} \varepsilon^{\frac{2s^{**}-2s^{*}}{2s^{**}+d}}
\]
Here $s^{**}>s^*$, $\tilde{\rho}>0$ and $p\in [1,\infty]$ are arbitrary,  
$\tilde{c}$ depends on $s^*,s^{**}$ and $p$, and the infimum is taken over 
arbitrary reconstruction methods $S$. 
We will choose $s^*=a$ and $s^{**}=s+a$. Moreover, we set $g^{\dagger}=F(f^{\dagger})$, 
$S=F\circ R$, and $ \rho= L_4\tilde{\rho}$ to obtain 
\[
\inf_{R} \sup_{\|f^\dagger-f_0\|_{s,p,\infty}\leq \rho} 
\mathbb{E}\left[\norm{F(f^{\dagger})-F(R(F(f^{\dagger})+\varepsilon Z))}
{B^{a}_{p,1}(\Omegaobs)}\right]
\geq \tilde{c} (\rho/L_4)^{\frac{2a+d}{2s+2a+d}} \varepsilon^{\frac{2s}{2s+2a+d}}\,.
\]
Together with \Cref{operator_additional} this yields the assertion 
with $c = (\tilde{c}/L_3) L_4^{-(2a+d)(2s+2a+d)}$. 
\end{proof}

\appendix
\section{Appendix}
\subsection{comparison to assumptions in the literature}\label{sec:comparison}
We give a characterization of condition \eqref{eq:strategy_operator} for a linear, bounded operator $\linOp \colon \Xspace\rightarrow \Yspace$. In the case $\Xspace=\ell^1$ the condition \eqref{eq:strategy_operator} is equivalent to source conditions studied in \cite{burger2013convergence}. 

\begin{theorem}\label{appendix:theorem}
Let $\linOp \colon \Xspace\rightarrow \Yspace$ be a bounded, linear operator between Banach spaces $\Xspace$ and $\Yspace$ and $(P_n:\Xspace\rightarrow \Xspace)_{n\in\mathbb{N}}$ a family of projections. Suppose $(\nu_n)_{n\in\mathbb{N}}$ is an increasing sequence of positive reals. Then the following statements are equivalent:
\begin{thmlist} 
\item $ \|P_n f\|_\Xspace\leq \nu_n \|\linOp f\|_\Yspace$ for all  $f \in \Xspace$ and $n\in\mathbb{N}$. \label{appendix:1}
\item For every $\xi\in \mathcal{R}(P_n^\ast)$ with $\|\xi\|_{\Xspace^\ast}=1$, there exists $\psi\in \Yspace^\ast$ with $\|\psi\|_{\Yspace^\ast}\leq \nu_n$ such that $\linOp^\ast\psi=\xi$.\label{appendix:2} 
\end{thmlist}
\end{theorem} 
\begin{proof}
\item[\itemref{appendix:1}$\Rightarrow$\itemref{appendix:2}:] 
Let $n\in\mathbb{N}$ and $\xi\in \mathcal{R}(P_n^\ast)$ with $\|\xi\|_{\Xspace^\ast}=1$.
Consider the linear map \linebreak $S_n\colon \mathcal{R}(\linOp )\rightarrow \Xspace$ given by $S_ng=P_n \linOp\inv g$. If $g\in \mathcal{R}(\linOp ),$ then 
\[ \|S_n g\|_\Xspace=\| P_n \linOp\inv g\|_\Xspace\leq \nu_n \|\linOp \linOp\inv g \|_\Yspace =  \nu_n \|g\|_\Yspace.\] 
Hence $\|S_n\|\leq \nu_n$. Therefore, $S_n$ extends to $\widetilde{S_n}\colon \overline{\mathcal{R}(\linOp )} \rightarrow \Xspace$ with $\|\widetilde{S_n}\|\leq \nu_n$.  Since $P_n$ is a projection, so is $P_n^\ast$. Hence $P_n^\ast \xi = \xi$. Let $\psi:= (\widetilde{S_n})^\ast\xi\in \overline{\mathcal{R}(\linOp )}^\ast$. Then
 \[ \|\psi\|_{\overline{\mathcal{R}(\linOp )}^\ast}= \|(\widetilde{S_n})^\ast\xi\|_{\overline{\mathcal{R}(\linOp )}^\ast}\leq \|(\widetilde{S_n})^\ast\| \|\xi\|_{\Xspace^\ast}\leq \|\widetilde{S_n}\|\leq \nu_n. \]
By the Hahn-Banach theorem $\psi$ can be extended to some $\tilde{\psi}\in \Yspace^\ast$ with $\|\tilde{\psi}\|_{\Yspace^\ast}\leq\nu_n$. Let $f\in \Xspace$. We compute 
\[ \langle \linOp^\ast \tilde{\psi} ,f \rangle = \langle \psi ,\linOp f\rangle =\langle (\widetilde{S_n})^\ast(\xi), \linOp f\rangle = \langle \xi, P_n \linOp\inv \linOp f\rangle= \langle \xi,P_n f\rangle=\langle P_n^\ast \xi, f\rangle = \langle\xi,f\rangle.\] 
This shows that $\xi=\linOp^\ast \tilde{\psi}\in \mathcal{R}(\linOp^\ast).$
\item[\itemref{appendix:2}$\Rightarrow$\itemref{appendix:1}:] 
Let $n\in\mathbb{N}$ and $f\in \Xspace$. By the Hahn-Banach theorem there exists 
$\xi\in \Xspace^\ast$ with $\|\xi\|_{\Xspace^\ast}=1$ such that $|\langle \xi, P_n f\rangle|=\|P_n f\|_\Xspace$. By assumption there exists $\psi\in \Yspace^\ast$ with $\|\psi\|_{\Yspace^\ast}\leq \nu_n$ such that $\linOp^\ast\psi=P_n^\ast\xi$ . We estimate 
\[ \|P_n f\|_\Xspace = |\langle \xi, P_n f\rangle|=|\langle P_n^\ast \xi,  f\rangle|=|\langle \linOp^\ast \psi,f\rangle|=|\langle \psi,  \linOp f\rangle|\leq \nu_n \|\linOp f\|_{\Yspace}.\qedhere\]
\end{proof}
\begin{remark}
As is \Cref{ell1example} we consider $\ell^1$ with $P_n$ the projection onto the first $n$ entries. We identify $(\ell^1)^\ast=\ell^\infty$. Then $P_n^\ast\colon \ell^\infty \rightarrow \ell^\infty$ is again the projection onto the first $n$ entries. We write $e_n\in \ell^\infty$ for the unit sequence with a $1$ in the $n$th entry and $0$'s else. Then $\mathcal{R}(P_n^\ast)=\spn\{e_1,\ldots,e_n\}$. Hence \Cref{appendix:theorem} reads as $\linOp \colon \ell^1\rightarrow \Yspace$ satisfies \eqref{eq:strategy_operator} if and only if $e_n\in \mathcal{R}(\linOp^\ast)$ for all $n\in \mathbb{N}$. This condition is assumed in \cite{burger2013convergence} to obtain convergence rates of $\ell^1$-regularization.
\end{remark}

\subsection{properties of the spaces \texorpdfstring{$B^s_{p,q}$}{}} 
In this section of collect some properties of the spaces $B^s_{p,q}$ used 
in this paper. Let us start with embeddings, which follow immediately from 
the well-known and straightforward embeddings of the sequence spaces 
$b^s_{p,q}$: 
\begin{proposition}\label{prop:embeddings}
The following inclusions hold true with continuous embeddings:
\begin{enumerate}[label=(\roman*),ref=(\roman*)]
\item\label{embed_mixed} 
%Let $p\in [1,2]$ and $a\geq d/p-d/2$. Then $\|\cdot \|_{-a,2,1}\leq \|\cdot\|_{0,p,1}$. Hence $B^0_{p,1}  \subset B^{-a}_{2,1}.$ 
$B^0_{p,1}  \subset B^{-a}_{\tilde{p},1}$ for $1\leq p\leq \tilde{p}\leq 2$ and $d/p-d/\tilde{p}\leq a$. 
\item \label{embed_q}
%Let $s\in \mathbb{R}$, $p, q_1, q_2 \in [1,\infty]$ with $q_1\leq q_2$. Then $\|\cdot \|_{s,p,q_2}\leq \|\cdot\|_{s,p,q_1}$.  Hence $B^s_{p,q_1}  \subset B^{s}_{p,q_2}.$ 
$B^s_{p,q_1}  \subset B^{s}_{p,q_2}$ for $s\in \mathbb{R}$, $p, q_1, q_2 \in [1,\infty]$ with $q_1\leq q_2$.
\item\label{embed_p} 
%Let $s\in \mathbb{R}$, $p_1,p_2, q\in [1,\infty]$ with $p_1\leq p_2$. Then $\|\cdot \|_{s,p_1,q}\leq C_\Lambda^{(1/p_1-1/p_2)}\|\cdot\|_{s,p_2,q}$. Hence $B^s_{p_2,q}  \subset B^{s}_{p_1,q}.$ 
$B^s_{p_2,q}  \subset B^{s}_{p_1,q}$ for  $s\in \mathbb{R}$, $p_1,p_2, q\in [1,\infty]$ with $p_1\leq p_2$.
\item \label{embed_s}
%Let $p\in [1,\infty]$ and $s>0$. Then $\|\cdot \|_{0,p,1}\leq (1-2^{-s})\inv \|\cdot\|_{s,p,\infty}$.  Hence $B^s_{p,\infty}  \subset B^0_{p,1}.$  
$B^s_{p,\infty}  \subset B^0_{p,1}$ for $p\in [1,\infty]$ and $s>0$.
\end{enumerate}
The norms of the corresponding embedding operators are bounded by $1$ in 
\ref{embed_mixed} and \ref{embed_q}, by $C_\Lambda^{(1/p_1-1/p_2)}$ in 
\ref{embed_p}, and by $(1-2^{-s})\inv$ in \ref{embed_s}.
\end{proposition}

%\begin{proof}
%Let $f\in \mathcal{D}^{\overline{s}}(\Omega)$. $\ref{embed_mixed}$ follows from $\|Q_jf\|_2\leq \|Q_jf\|_p$ for all $j\in\nat.$  The inequality $\|\cdot\|_{\ell^{q_2}}\leq \|\cdot\|_{\ell^{q_1}}$ implies $\ref{embed_q}.$ $\ref{embed_p}$ follows from $\|Q_jf\|_{p_1}\leq |\Lambda_j|^{(1/p_1-1/p_2)} \|Q_j\|_{p_2}$ and $|\Lambda_j|\leq C_\Lambda 2^{jd}$ for all $j\in \nat.$ To prove the last statement we estimate
%\[ \|f\|_{0,p,1}
%= \sum\nolimits_{j\in\nat}  2^{-js}2^{js}  2^{jd(\frac{1}{2}-\frac{1}{p})} \|Q_jf\|_p 
%\leq \left(\sum\nolimits_{j\in\nat} 2^{-js}\right) \|f\|_{s,p,\infty}.\qedhere
%\]
%\end{proof}
%Let $p,q\in [1,\infty]$ and $s\in \mathbb{R}$. Recall that the Besov sequence space $b^s_{p,q}:=b^s_{p,q}(\Lambda)$ is given by 
%\[ \left\{ (z_\lambda)_{\lambda\in \Lambda} \colon  \left\| \left(2^{js}2^{jd(\frac{1}{2}-\frac{1}{p})} \|(z_{j,k})_{k\in \Lambda_j}\|_p\right)_{j\in \nat}\right\|_{\ell^q} < \infty \right\}. \]  
\begin{proposition}\label{compactness_embedding}
Suppose \eqref{eq:completeness} and $1\leq p \leq\tilde{p}\leq 2$ 
let $d/p-d/\tilde{p}<a<\overline{s}$. Then the embedding $B^0_{p,1}\subset B^{-a}_{\tilde{p},1}$ is compact.
\end{proposition} 
\begin{proof}
By the definition of the space $B^{-a}_{\tilde{p},1}$ the map 
\( Q\colon B^{-a}_{\tilde{p},1}\rightarrow b^{-a}_{\tilde{p},1}\) 
%\quad \text{given by}\quad S(f)= (\langle f, \dframe_\lambda\rangle)_{\lambda\in \Lambda}\]  
is isometric. The completeness of $B^{-a}_{\tilde{p},1}$ 
(see \eqref{eq:completeness}) 
implies that $Q(B^{-a}_{\tilde{p},1})$ is closed in $b^{-a}_{\tilde{p},1}$. The emdedding $b^0_{p,1}\subset b^{-a}_{\tilde{p},1}$ is compact since it is a limit of finite rank projection operators. Hence the embedding $\iota \colon Q(B^0_{p,1})\subset  Q(B^{-a}_{\tilde{p},1})$ is compact. We conclude the compactness of embedding $B^0_{p,1}\subset B^{-a}_{\tilde{p},1}$ by its factorization as $Q\inv \circ \iota\circ Q$.
\end{proof}

To characterize dual spaces, let $p^\prime\in[1,\infty]$ denote the H\"older conjugate 
of $p\in [1,\infty]$, i.e.\ $\frac{1}{p}+\frac{1}{p^\prime}=1$, and analogously 
for $q$. From the well-known duality $(b^s_{p,q})' = b^{-s}_{p^\prime,q^\prime}$ 
we obtain:
\begin{proposition} \label{dual_spaces_iso}
Let $p,q\in [1,\infty)$ with H\"older conjugates $p^\prime$ and $q^\prime$, 
 respectively,  let $|s|< \overline{s}$ and suppose that \eqref{eq:Q_surjective} 
holds true. Then
\[
(B^s_{p,q})' = B^{-s}_{p^\prime,q^\prime}.
\]
%Suppose that $(\dframe_\lambda)_{\lambda\in \Lambda}$ is a Riesz basis of $L^2(\Omega)$. Let $p,q\in [1,\infty)$ with H\"older conjugates $p^\prime$ respectively $q^\prime$ and $|s|< \overline{s}$. Then the map 
%$ \Psi\colon b^{-s}_{p^\prime,q^\prime}\rightarrow (B^s_{p,q})^\prime $ given by  
%\[ \langle\Psi z, f\rangle = \sum_{(j,k)\in \Lambda} z_{j,k} \langle \dframe_{j,k}, f\rangle \quad\text{for } z\in  b^{-s}_{p^\prime,q^\prime} \text{ and } f\in B^s_{p,q}\]
%is an isometric isomorphism. 
\end{proposition} 
\subsection{An optimality condition}
\begin{lemma}\label{lem:optimality}
Let $D\subset\Xspace$ be a non-empty, open subset of a Banach space $\Xspace$, 
let $\mathcal{F}:D\to \mathbb{R}$ be Gateaux-differentiable, $\mathcal{G}:\Xspace \to\mathbb{R}$ convex and continuous, 
and let $\overline{f}\in D$ be a local minimum of $\mathcal{F}+\mathcal{G}$. Then 
\[
0\in \mathcal{F}'[\overline{f}] + \partial \mathcal{G}(\overline{f}).
\]
\begin{proof}
Recall that the first variation of a functional $\mathcal{H}:D\to\mathbb{R}$ 
at $\overline{f}\in D$ in direction $v\in \Xspace$ is defined by 
\(
\delta \mathcal{H}(\overline{f};v):=\lim_{t\searrow 0}\frac{1}{t} 
\left(\mathcal{H}(\overline{f}+tv)-\mathcal{H}(\overline{f})\right)
\)
if the limit exists. As $\mathcal{F}$ is Gateaux differentiable, we have 
$\delta\mathcal{F}(\overline{f};v)=\mathcal{F}'[f]v$ for all $v$. Moreover, as 
$\mathcal{G}$ is convex and finite, $\delta\mathcal{G}(\overline{f};v)$ 
also exists for all $v$ (see \cite[Prop.~42.5, Thm.~47.C.]{zeidler3}).
%\[
%\delta\mathcal{G}(\overline{f};v)\geq \sup\nolimits_{f^*\in \partial G(\overline{f})}
%\langle f^*,v\rangle.
%\]
It follows that $\delta(\mathcal{F}+\mathcal{G})(\overline{f};v) 
= \delta \mathcal{F}(\overline{f};v)+\delta 
\mathcal{G}(\overline{f};v)$ exists, and since $\overline{f}$ is a local minimum, we 
have $\delta(\mathcal{F}+\mathcal{G})(\overline{f};v)\geq 0$ for all $v\in\Xspace$. 

Suppose now that the assertion is false and 
$-\mathcal{F}'[\overline{f}]\notin \partial \mathcal{G}(\overline{f})$. 
Let us define the functional 
\[
\mathcal{H}(f):= \mathcal{F}(\overline{f})+\mathcal{F}'[\overline{f}](f-\overline{f})
+ \mathcal{G}(f),\qquad f\in\Xspace
\]
which is convex and continuous by the assumptions of the lemma. 
Moreover, our assumption 
$-\mathcal{F}'[\overline{f}]\notin \partial \mathcal{G}(\overline{f})$ 
is equivalent by the sum rule to $0\notin \partial \mathcal{H}(\overline{f})$. 
Due to the equivalences in \cite[Thm.~47.C.]{zeidler3}
\[
\overline{f}\in\argmin_{f\in\Xspace} \mathcal{H}(f) \;\Leftrightarrow\;
\inf_{v\in\Xspace} \delta \mathcal{H}(\overline{f};v)\geq 0 
\;\Leftrightarrow\; 0\in\partial \mathcal{H}(\overline{f})
\]
there exists $v\in\Xspace$ such that $\partial\mathcal{H}(\overline{f};v)< 0$. 
This implies $\delta(\mathcal{F}+\mathcal{G})(\overline{f};v)
= \delta\mathcal{H}(\overline{f};v)< 0$, contradicting the 
assumption that $\overline{f}$ is a local minimum of $\mathcal{F}+\mathcal{G}$. 
%Moreover, 
%$\delta \mathcal{H}(\overline{f};v)=\delta (\mathcal{F}+\mathcal{G})(\overline{f};v)$ 
%for all $v\in\Xspace$. 
\end{proof}
\end{lemma}

\paragraph*{acknowledgement:} 
We would like to thank Gerlind Plonka-Hoch for helpful discussions. 
Financial support by DFG through grant RTG 2088 is gratefully acknowledged. 
\bibliographystyle{abbrv}
\bibliography{lit} 	
\end{document}